\documentclass[11pt]{amsart}
 \usepackage[foot]{amsaddr}
\usepackage{bbm,bm}
\usepackage{amssymb}
\usepackage[margin=1.5in, total={6.0in, 8in}]{geometry}
\usepackage{colordvi}
\usepackage{graphicx,xspace}
\usepackage{color}
\usepackage{xspace}

\usepackage[bookmarks=false,colorlinks,linkcolor=blue,citecolor=green]{hyperref}

\definecolor{mains}{cmyk}{.3, .85, .75, 0}  
\definecolor{afb}{rgb}{0.03, 0.27, 0.49}
\definecolor{def}{rgb}{0.27, 0.03, 0.49}

\newcounter{FNC}[page]
\def\fauxfootnote#1{{\addtocounter{FNC}{2}$^\fnsymbol{FNC}$%
     \let\thefootnote\relax\footnotetext{$^\fnsymbol{FNC}$\Magenta{#1}}}}

\numberwithin{equation}{section}
\newtheorem{theorem}{Theorem}[section]

\newtheorem{lemma}[theorem]{Lemma}

\newtheorem{thm}[theorem]{Theorem}
\newtheorem{defn}[theorem]{Definition}
\newtheorem{exm}[theorem]{Example}
\newtheorem{rem}[theorem]{Remark}
\newenvironment{definition}[1][]{\rm\begin{defn}[#1]\rm}{\end{defn}}

\newenvironment{remark}[1][]{\rm\begin{rem}[#1]\rm}{\end{rem}}

\author{Stefan Forcey} \address[S. Forcey]{
    Department of Mathematics\\
    The University of Akron\\
    Akron, OH 44325-4002
    }
    \email{sforcey@uakron.edu}  \urladdr{http://www.math.uakron.edu/\~{}sf34/}
\author{Drew Scalzo} \address[D. Scalzo]{
    Department of Mathematics\\
    The University of Akron\\
    Akron, OH 44325-4002
    }

\title[Phylogenetic resistance distance]{Phylogenetic networks as circuits with resistance distance}

\begin{document}

\begin{abstract}
    
    Phylogenetic networks are notoriously difficult to reconstruct. Here we suggest that it can be useful to view unknown genetic distance along edges in phylogenetic networks as analogous to unknown resistance in electric circuits.  This \emph{resistance distance}, well known in graph theory, turns out to have nice mathematical properties which allow the precise reconstruction of networks. Specifically we show that the resistance distance for a weighted 1-nested network is Kalmanson, and that the unique associated circular split network fully represents the splits of the original phylogenetic network (or circuit). In fact, this full representation corresponds to a face of the balanced minimal evolution polytope for level-1 networks. Thus the unweighted class of the original network can be reconstructed by either the greedy algorithm neighbor-net or by linear programming over a balanced minimal evolution polytope. We begin study of 2-nested networks with both minimum path and resistance distance, and include some counting results for 2-nested networks.
    

 \end{abstract}
\maketitle



\section{Introduction}

Consider an electrical circuit: a network made of wires joining resistors in parallel and in sequence, with some portion hidden inside an opaque box. It is not always possible to determine that portion by testing the visible leads. However, we prove here that if the hidden portion has a particular form made of  connected cycles, and we can test the resistance between all the pairs of leads, then the lengths and connected structure of the cycles in the circuit are uniquely determined. The mathematics used to recover that circuit is more typically found in work on phylogenetic networks.

 Modeling heredity as the flow of genetic information suggests that mutations in DNA might be analogous to resistance in an electrical circuit.
The weights of edges in a phylogenetic network can represent genetic distances: if we have the genomes of the two endpoints of an edge then we can use a model of mutation rates to calculate a real number distance. For several edges that form the unique path between two taxon-labeled leaves, the total distance is the sum of those edge weights. Paths between leaves are only unique if the network is a tree. 
When paths between are not unique, one option is to take the distance to be that of the minimum length path. This option may correspond to a parsimonious approach---assuming the least complicated history. This minimum path length distance is studied for instance in \cite{scalzo}. 

\begin{figure}
    \centering
    \includegraphics[width=\textwidth]{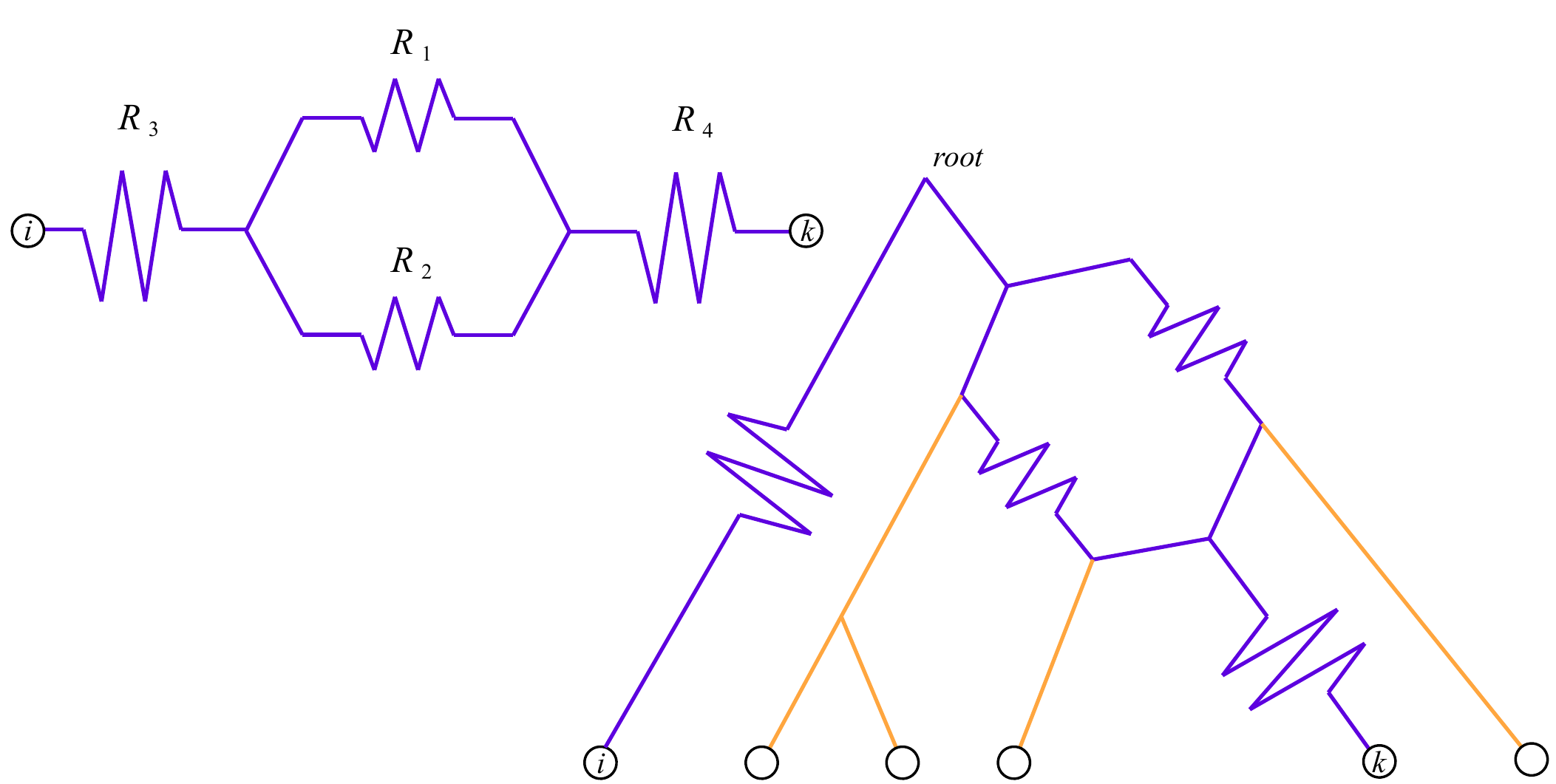}
    \caption{The total resistance from $i$ to $k$ is $\displaystyle{R_{i,k} = \frac{R_1R_2}{R_1+R_2}+R_3+R_4}$. On the left is the circuit itself, on the right we see it as a pairwise circuit within a phylogenetic network. Here we have chosen $i$ to be an outgroup, so the network is rooted at the top and the downward direction is forward in time.}
    \label{circuit1}
\end{figure}

Instead, however, a greater weight of an edge could represent a greater loss of information. Dividing and rejoining of edges illustrates events such as speciation, recombination, or hybridization. If the genetic information of an ancestor genome can be shared among descendants, and then collaboratively recovered upon hybridization, then a different metric than minimum distance may be appropriate.    
Here we consider weighted phylogenetic networks with the \textcolor{def}{\emph{resistance distance}}, or \textcolor{def}{\emph{resistance metric.}} The distance between two leaves of the network is found by considering the edge weights as  electrical resistance, obeying Ohm's law. The metric resistance distance for all nodes (not only leaves) of a graph is introduced in \cite{klein93}, and studied closely in subsequent papers such as \cite{klein15} and \cite{klein19}. To study graphs, the resistance of each edge is often assumed to have unit value, but the definitions allow any weight. We review the definitions in Section~\ref{sec_weight}. 

In \cite{curtis2}, \cite{curtis0} and \cite{curtis1},  the authors study circular planar graphs with boundary nodes that are analogous to the leaves of our phylogenetic networks. They consider resistance values (or conductivity) on the edges. They prove that complete information about the linear map which transforms electric current values at each boundary node to electric current values at all the edges can be used in some cases to recover the resistance values. In our applications there is no way to know the complete map of boundary currents to edge currents. However, we seek only to recover the graphical structure of the network, not the original edge weights.  

In \cite{filar}, the authors consider the entire set of resistance distances (again using unit values for edges), between any pair of nodes (not only leaves.) They show that using this metric is useful for discovering Hamiltonian cycles via algorithms for the Travelling Salesman problem. There is a close connection to our applications, since the algorithm neighbor-net can be used as a greedy approach to the Travelling Salesman problem as shown in \cite{Pachter2}.

\subsection{Main Results and Overview}
In Section~\ref{defs} we start by reviewing Ohm's law and resistance distance. Then we review the relevant definitions of mathematical phylogenetics, many taken from other sources to help make this paper self-contained.  In Section~\ref{thms} we state and prove the main results for 1-nested phylogenetic networks $N$. The upshot is that when the distances between taxa are effective resistances based on unknown connections, then using well known methods  we can recover an unweighted circular split network, which gives us the precise class of (unweighted) 1-nested phylogenetic network. Specifically, this recovery is via the (greedy) algorithm neighbor-net as decribed in Theorem~\ref{net} or  linear programming; see Theorem~\ref{face}. 

Several features of the resistance distance seem exactly suited to phylogenetic networks with weighted edges. First, from Theorem~\ref{kal}, the resistance distance of 1-nested phylogenetic networks is Kalmanson, allowing the circular split network to be uniquely reconstructed from the measured distances. Second, from Theorem~\ref{tsplits}, that reconstructed circular split network always displays precisely the same  splits as the original network. As a consequence, the trivial splits which are the traditional final edges to the leaves of a phylogenetic network are automatically guaranteed to be represented in the split network---this is a condition beyond the basic Kalmanson condition. Finally, triangular subgraphs are interchangeable with three-edge stars when measuring resistance distance. This is known as the Y-$\Delta$ transform, pictured in Figure~\ref{wye}. The Y-$\Delta$ equivalence mirrors the fact that triangles in a phylogenetic network, when attached via bridges to the rest of the network, are indistinguishable from degree-three tree-like vertices by the linear functionals used for balanced minimal evolution. As well, the split networks are bipartite, so triangle free.

In Section~\ref{sec_poly} we review the balanced minimal evolution polytopes, and show how our results can be interpreted geometrically, in Theorems~\ref{vert} and~\ref{face}. In Section~\ref{counter} we point out some interesting counterexamples and limiting cases, and conjecture about how to extend our results to more complicated networks. Section~\ref{too} contains some new results on 2-nested networks with regard to the minimum path distance. Finally in Section~\ref{jukes} we consider qualifications of experimental distance measurements in phylogenetics that would give justification for assuming the resistance analogy to be valid in practice.

\section{Definitions and cited results}\label{defs}
We start by reviewing some equations from electric circuit theory.
\subsection{Electricity} Given a conductive circuit with a power supply, the materials have resistance $R$ and the power causes a current $I.$ 
The classic  Ohm and Kirchoff equations include: $R = V/I$ and $I = I_1+I_2.$ The first depends on the conductive material---it must be experimentally verified. It relates the resistance in a circuit to the constant voltage drop over the circuit and the constant current in all of the circuit.  The second states that total current must equal the sum of  circuit-parallel portions of that current after a branching in the circuit.  Together, these rules imply the law for total resistance $R_T$ for a pair of circuit-parallel resistances $R_1,R_2.$ We have $R_T = R_1R_2/(R_1+R_2),$ which we refer to as Ohm's law for parallel resistance. Also, the voltage drop over a closed circuit must equal the total voltage: this implies that resistors in series are summed to find the total resistance. We illustrate the basic calculation of a total resistance in Figure~\ref{circuit1}. We illustrate the implied $Y$-$\Delta$ equivalence in Figure~\ref{wye}.
\begin{figure}
    \centering
    \includegraphics[width=5in]{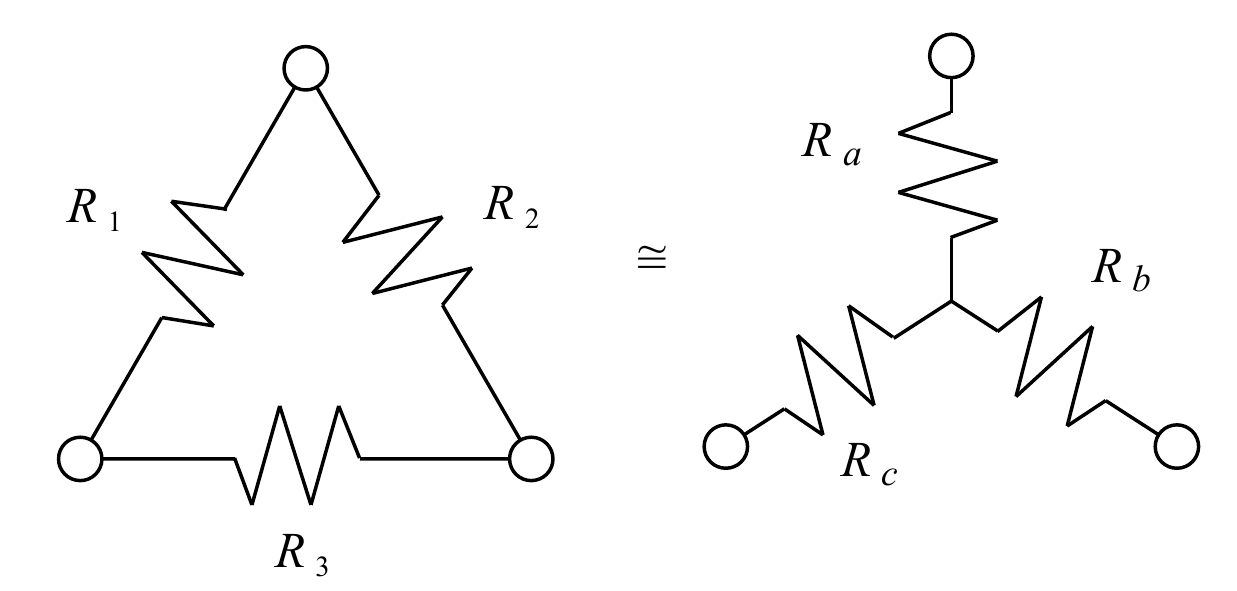} 
    \caption{The two networks shown here have identical resistance between any two corresponding pairs of nodes at the three corners. Here $\displaystyle{R_a = \frac{R_1R_2}{R_1+R_2+R_3}, R_b=\frac{R_2R_3}{R_1+R_2+R_3}}$, and $\displaystyle{R_c = \frac{R_1R_3}{R_1+R_2+R_3}.}$}
    \label{wye}
\end{figure}

\subsection{Phylogenetic definitions}
Many of the definitions and notes here are repeated (sometimes verbatim) from \cite{scalzo} for the sake of self-containment. For further reference, see \cite{steelphyl} and \cite{gambette-huber}.

A \textcolor{def}{\emph{split}} $A|B$ is a bipartition of $[n] = \{1,\dots,n\}.$ That is, $A$ and $B$ are non-empty disjoint subsets whose union is $[n].$  The two parts of a split are often called \textcolor{def}{\emph{clades}}. If one clade of a split has only a single element, we call that split \textcolor{def}{\emph{trivial.}} A \textcolor{def}{\emph{split system}}  is a set $s$ of splits of $[n]$ which contains all the trivial splits. We say a split system $s$  \textcolor{def}{\emph{refines}} another split system $s'$ when $s \supset s'$.
In this paper all graphs are simple (no multi-edges) and connected.
\begin{definition}
An  \textcolor{def}{\emph{(unrooted) phylogenetic network}} on $[n]$ is a simple connected graph with:
\begin{enumerate}
    \item[i.] Labeled leaves: $n$ degree-1 vertices,  labeled bijectively with the elements of $[n]$,
    \item[ii.] Unlabeled nodes: all these must have degree larger than 2.
\end{enumerate}
For the remainder of the paper, all phylogenetic networks are assumed to be unrooted and without any edge directions.
\end{definition} 

A split $A|B$ is \textcolor{def}{\emph{displayed}} by a phylogenetic network $N$ when there is (at least) one subset of edges of $N$ whose deletion (keeping all nodes) results in two connected components  with $A$ and $B$ their respective sets of labeled leaves. We call that collection of edges a \textcolor{def}{\emph{minimal cut}} displaying the split when the collection contains no proper subset displaying the same split.  A \textcolor{def}{\emph{bridge}} is a single edge which displays a split. A \textcolor{def}{\emph{trivial bridge}} displays a trivial split. A \textcolor{def}{\emph {phylogenetic tree}} is a cycle-free phylogenetic network, so every edge is a bridge. 
Figure~\ref{splits} shows examples of splits displayed, for the trees and their two generalizations described here: phylogenetic networks and split networks. 
\begin{figure}
    \centering
   \includegraphics[width=\textwidth]{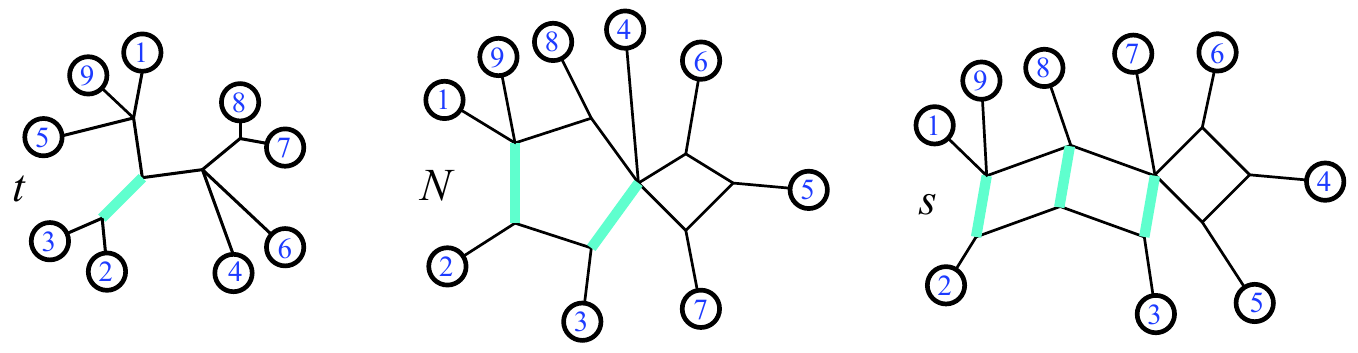}
    \caption{Modified from a figure in \cite{scalzo}. In a phylogenetic tree $t$, on the left, splits are always single edges. The highlighted edge is the split $\{2,3\}|\{1,4,5,6,7,8,9\}.$ That same split is a pair of edges making a minimal cut in the 1-nested phylogenetic network $N$, center. Finally on the right, that same split is a set of parallel edges in a circular split network $s$.}
    \label{splits}
\end{figure}
 Recall that a cycle in a graph is a path of edges that does not revisit any nodes except for the node at which it starts and ends. The following
 is defined in \cite{gambette-huber}: 
 \begin{definition}
 An unrooted phylogenetic network $N$ is called \textcolor{def}{\emph{1-nested}}  when each edge of $N$ is contained in at most one cycle, and $N$ is triangle-free---all cycles are of length greater than 3 edges. 
 \end{definition}

A 1-nested phylogenetic network can be drawn in the plane with its leaves on the exterior, which is referred to as outer planarity.  We consider two 1-nested networks to be \textcolor{def}{\emph{split-equivalent}} if they display the same set of splits.  See Figure~\ref{fig:mo_examplo} for examples.  \textcolor{def}{\emph{Twisting}} a phylogenetic network around a bridge (reflecting one side through the line of the bridge), or around a cut-point node, does not change the list of splits.  Any cyclic order of the leaves seen around the exterior in some representative drawing of a 1-nested phylogenetic network is said to be \textcolor{def}{\emph{consistent}} with that split system. Figure~\ref{processo} shows examples. 
\begin{figure}
    \centering
    \includegraphics[width=\textwidth]{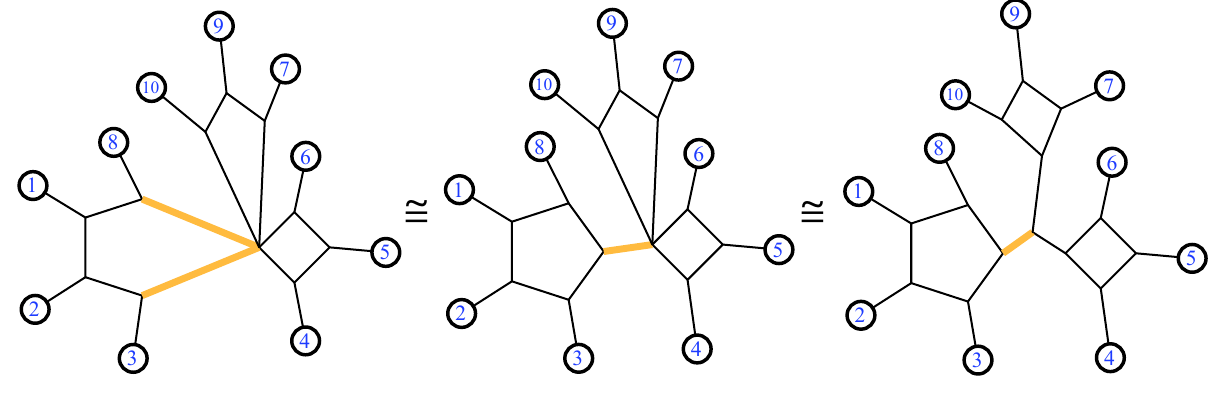}
    \caption{A trio of equivalent 1-nested phylogenetic networks. All display the same set of splits. The highlighted edges display the same split in each network.}
    \label{fig:mo_examplo}
\end{figure}
 A \textcolor{def}{\emph{binary}} phylogenetic network is one in which the unlabeled nodes each have degree 3. A phylogenetic network $N$ refines another, written $N\ge M$, when the splits displayed by $M$ are a subset of those displayed by $N.$ Several of these terms are exhibited in Figure~\ref{processo}.
Next we review the definition of another generalization of a phylogenetic tree.
\begin{definition}
A  \textcolor{def}{\emph{split network}} displaying a split system $s$ on $[n]$ is an embedding in Euclidean space of a simple connected graph, also called $s$, with the following:
\begin{enumerate}\item[i.] Labeled leaves: $n$ degree 1 nodes are bijectively labeled by $[n]$.
\item [ii.] Unlabeled nodes: these have degree larger than 1.
\item[iii.] A partition of the set of edges:  the parts of this partition are called \emph{split-classes}. There is one split-class for each split $A|B$ in the system. It is required that for any two leaves, the set of edges on a shortest path between them intersects each split-class in at most one edge, and that the set of splits thus traversed is the same for any shortest path between those two leaves.
\item[iv.] The split-class of edges corresponding to a split $A|B$ comprises a minimal cut displaying that split:  deletion of those edges  results in two connected components with respective labeled leaves $A$ and $B$. 
\end{enumerate}
\end{definition}
The resulting bipartite graph is often shown with each class of edges  embedded as a set of equal length parallel line segments. (Note: here parallel means geometrically parallel.) Alternate definitions use colors;  the edges in a split-class are colored alike, as in, \cite{basic}, \cite{steelphyl}. A split-class of size one is a bridge. Two split networks are defined to be equivalent when they represent the same split system. 

\begin{definition}
A \textcolor{def}{\emph{circular split system}} is a split system which allows the embedding of a representative split network in the plane, with the labeled nodes all on the exterior, and thus arranged in a circular order. We refer to these representatives as circular split networks.
\end{definition}
  Just as for phylogenetic networks, \textcolor{def}{\emph{twisting}} a circular split network around a bridge (reflecting one side through the line of the bridge), or around a cut-point node, does not change the list of splits.  Any cyclic order of the leaves seen  in some representative circular split network is said to be \textcolor{def}{\emph{consistent}} with that split system. 
 Two circular split networks are equivalent if they display the same split system. For instance see Figure~\ref{fig:mo_examplo_circ}. 
\begin{figure}
    \centering
    \includegraphics[width=6in]{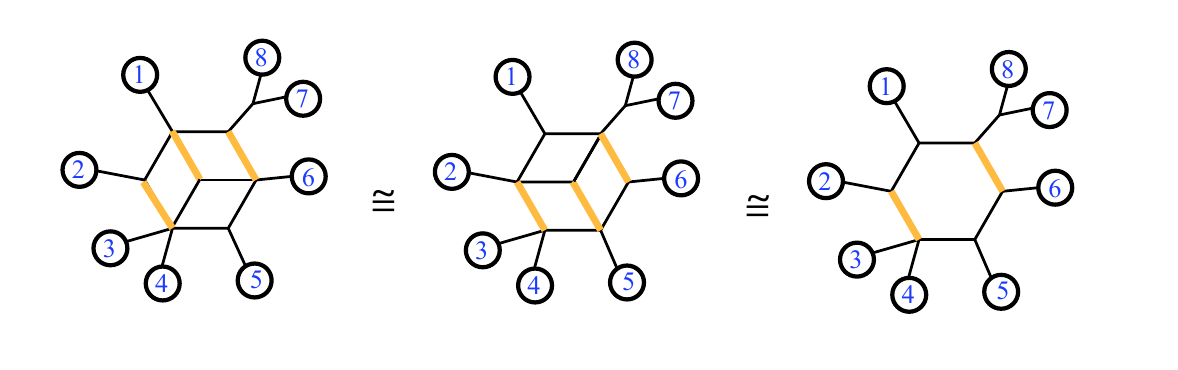}
    \caption{Modified from a figure in \cite{scalzo}. A trio of equivalent split networks. All three represent the same split system. The highlighted edges display the same split in each network. The third is the invariant exterior subgraph of all three.}
    \label{fig:mo_examplo_circ}
\end{figure}
The following lemma is from \cite {scalzo}, included here without proof for the terminology that will be useful in the next section.
\begin{lemma}
Given a circular split network $s$,
the nodes and edges adjacent to the exterior of the graph are a subgraph which is invariant: that is, this \textcolor{def}{\emph{exterior subgraph}} will be identical to the exterior subgraph of any circular split network representing the same set of splits as $s$. 
\end{lemma}
Again for example see Figure~\ref{fig:mo_examplo_circ}.  
Introduced in \cite{scalzo} is a subclass of circular split networks. 
\begin{definition}\cite{scalzo}
An \textcolor{def}{\emph{outer-path circular split system}} is a split system whose representative circular split networks have shortest paths between pairs of leaves which can all be chosen to lie on the exterior of the diagram, that is, using only edges adjacent to the exterior. 
\end{definition}
 For examples, see Figure~\ref{shorty}.
\begin{figure}
    \centering
    \includegraphics[width=\textwidth]{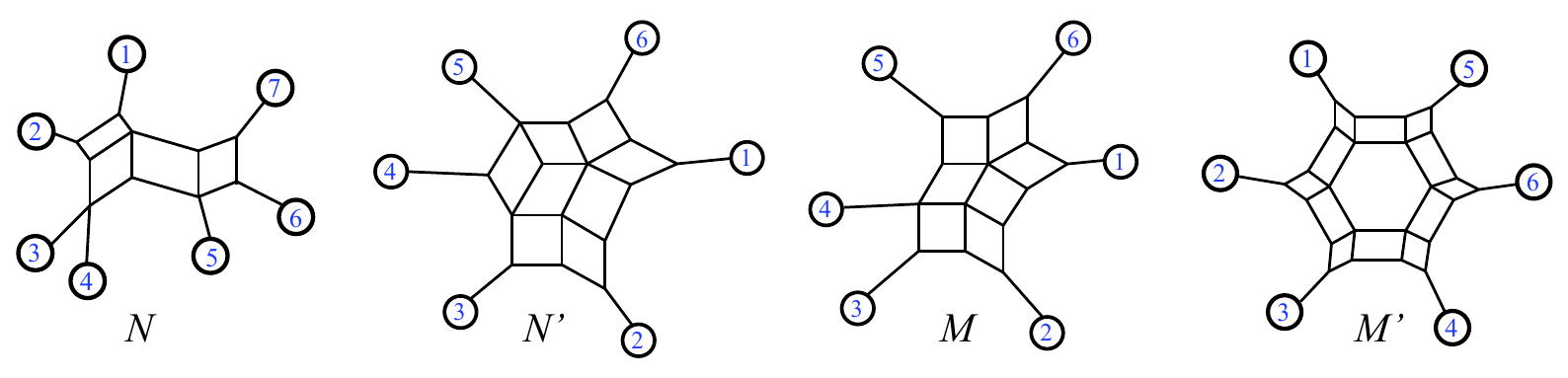}
    \caption{\cite{scalzo} From the left, $N$ and $N'$ are  outer-path circular split networks. In contrast $M$ and $M'$ are non-outer-path circular split networks.}
    \label{shorty}
\end{figure}

\subsubsection{Functions for unweighted networks}\label{functions}
The definitions in this section are repeated from \cite{scalzo}, but originate in \cite{gambette-huber}.

\begin{definition}\label{sig}
For a 1-nested phylogenetic network $N$  define $\Sigma(N)$ to be the circular split system made up of the splits displayed by $N.$ Thus the map $\Sigma$ takes a 1-nested phylogenetic network and outputs the set of splits displayed by $N.$ \end{definition} 

In \cite{gambette-huber} it is shown that $\Sigma(N)$ is  a circular split system, since it can be represented by a circular split network, also referred to as $\Sigma(N)$.
Examples of representations of $\Sigma(N)$ are seen in Figure~\ref{fig:sigma_map}. Note that since the bridges in a split network are invariant, every representation of $\Sigma(N)$ will have the same bridges: these will match the maximal set of bridges of any representation of $N.$ The range of $\Sigma(N)$ will be referred to as the \textcolor{def}{\emph{faithfully phylogenetic}} circular split networks.

From \cite{scalzo} and \cite{gambette-huber}, we repeat an algorithm for drawing a circular split network to represent $\Sigma(N).$ Each split of $N$ must correspond to a class of parallel edges in $\Sigma(N).$ The simplest representing network  would just subdivide the edges of $N$ to make a class for each split, but we show how to construct a representative which makes the splits more visible via bridges and parallelograms.  For $m \ge 5$, each $m$-cycle in $N$ is replaced by an $m$-\emph{marguerite}: a collection of exactly $m^2-4m$ parallelograms arranged in a circle, each sharing sides with two neighbors, specifically organized as follows: each node of the original $m$-cycle is replaced by a rhombus, and then each edge of the cycle is replaced by $m-5$ parallelograms in a row. The  rows are attached to the  rhombi along adjacent edges of each rhombus, so that the whole arrangement has $m(m-5)$ sides on the interior of the original $m$-cycle, and $m(m-3)$ sides on the exterior. Bridges are attached to the $m$ remaining degree-2 vertices, one at each of the rhombi that replaced the original $m$ nodes of the cycle.
\begin{figure}
    \centering
    \includegraphics[width=\textwidth]{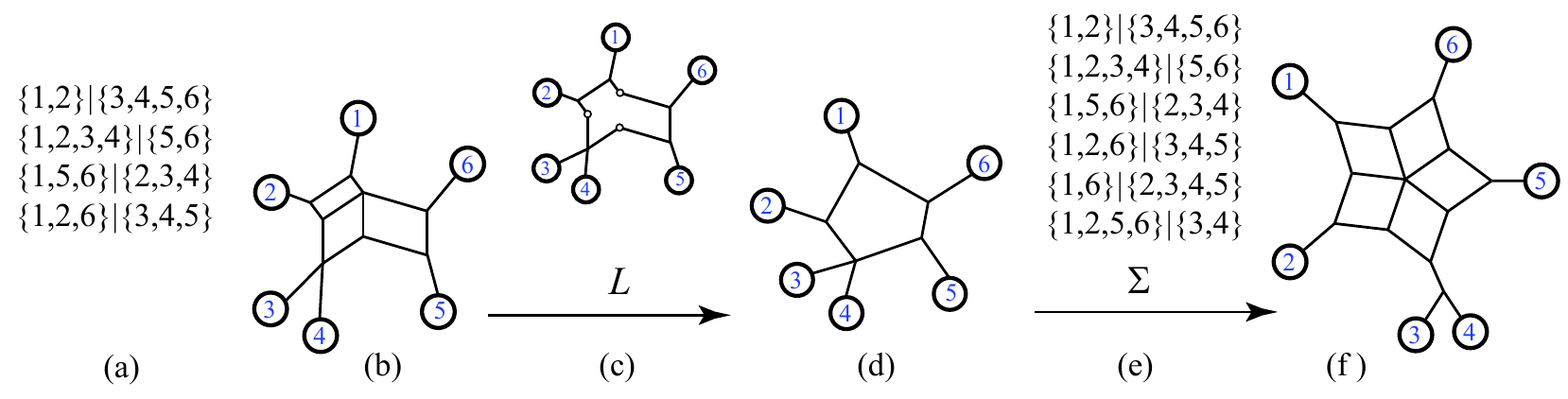}
    \caption{Here we see (a) a split system $s$, with only the non-trivial splits listed (trivial splits are assumed to be included), (b) a circular split network representing $s$, (c) the exterior subgraph of $s$ as a step in the process of applying $L$, (d) the output 1-nested phylogenetic network $N = L(s)$, (e) the split system $\Sigma(N)$  displayed by $N$, again showing only the non-trivial splits, and (f) a representative circular split network also referred to as $\Sigma(N)$. We see that $\Sigma(N) \ge s$, and that the cyclic orders $(1,2,3,4,5,6)$ and $(1,2,4,3,5,6)$ are both consistent with $N$ and with $s.$}
    \label{processo}
\end{figure} 
\begin{figure}
    \centering
    \includegraphics[width=5in]{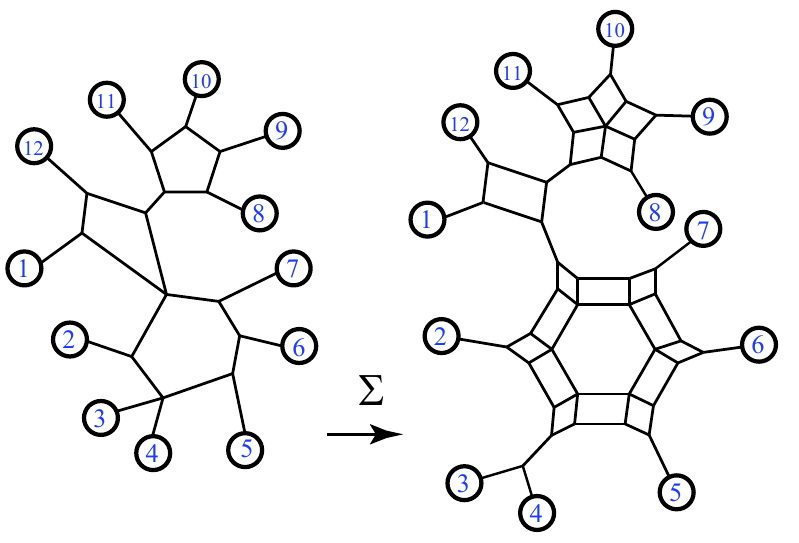}
    \includegraphics[width=5.25in]{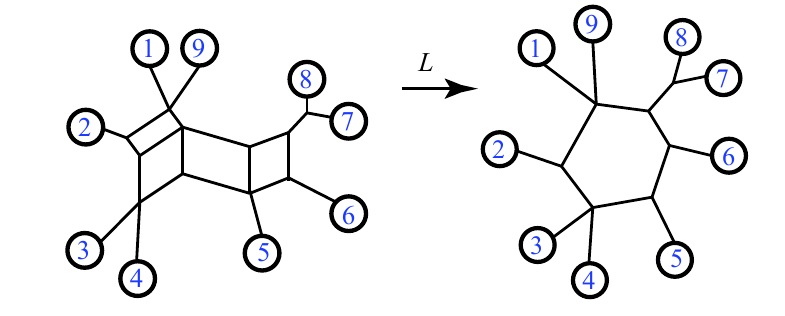}
    \caption{Examples of the functions $\Sigma$ and $L$.}
    \label{fig:sigma_map}
\end{figure}

Now for a function that takes circular split networks to  1-nested phylogenetic networks.  This function is shown to exist in \cite{gambette-huber}, and described on the split networks which are images of the function $\Sigma.$ In \cite{durell} and \cite{scalzo} we define the general function $L$ as follows: 
\begin{definition}
  For a circular split system $s$, define $L(s)$ to be the \emph{smoothed exterior subgraph} of a representative split network $s.$ Thus $L$ takes a circular split system (with a given representation) and outputs a 1-nested phylogenetic network. The operation of $L$ is easy to describe as 1) erasing the interior edges of split network $s$ and 2) smoothing, which here refers to removing any degree-2 nodes that are seen in the exterior subgraph. Such a node is removed, but the two edges adjacent to it are joined to form a single edge.  \end{definition}
    
  Recall that the nodes and edges adjacent to the exterior of a circular split network are an invariant subgraph for the split system, so the function $L$ is well-defined on split systems.  Examples exhibiting $\Sigma$ and $L$ are in Figure~\ref{processo} and Figure~\ref{fig:sigma_map}.


\begin{remark}\label{reflection}~
Note that by its construction, $L$ preserves bridges and cut-point nodes. When restricted to phylogenetic trees, the functions $L$ and $\Sigma$ are both the identity. In \cite{scalzo} several other properties of the two functions are listed, in the process of showing that $L$ and $\Sigma$ form a Galois reflection, as in \cite{primer}. These include the facts that $L$ is surjective but not injective,  $\Sigma$ is injective but not surjective, and  and that $L\circ\Sigma$ is the identity map. 
\end{remark}

\subsubsection{Weights and metrics} We continue to repeat definitions from \cite{scalzo}
Weighted networks can be constructed in two distinct ways: by assigning non-negative real numbers to splits or to edges.  
\begin{definition}
A \textcolor{def}{\emph{weighted phylogenetic network}} $N$ has non-negative real numbers assigned to its edges, described  by a weight function $w_N.$ 
\end{definition}

\begin{definition}
A \textcolor{def}{\emph{weighted split network}} $s$ has non-negative weights assigned to each split, by a weight function $w_s$. Equivalently, $w_s$ assigns a weight to every edge, with the requirement that each edge in a (geometrically parallel) split-class of $s$ has the same weight.
\end{definition}

\begin{definition}
For a weighted phylogenetic network $N$,  or a weighted split network $s$, we denote by $\overline{N}$, respectively $\overline{s}$, the unweighted networks found by forgetting the weights. 
\end{definition}

As in \cite{scalzo}: a pairwise distance function assigns a non-negative real number to each pair of values from $[n]$. We call the lexicographically listed outputs for distinct pairs a \textcolor{def}{\emph{distance vector}}  $\mathbf{d}$, with entries denoted $d_{ij} = \mathbf{d}(i, j)= \mathbf{d}(j, i)$ for each pair of taxa $i \neq j \in [n]$ (also known as a dissimilarity matrix, or discrete metric when obeying the metric axioms.)

\begin{definition}
When the distance vector is \textcolor{def}{\emph{Kalmanson}}, or \textcolor{def}{\emph{circular decomposable}} it means there exists a cyclic order of $[n]$ such that for any subsequence $(i,j,k,l)$ of that order, $\mathbf{d}$ obeys this condition: $$\max\{d_{ij}+d_{kl}, d_{jk}+d_{il}\} \le d_{ik}+d_{jl}.$$
\end{definition}

 \begin{definition}
Given  a weighted split system $s$ on $[n]$ we can derive a metric $\mathbf{d}_s$ on $[n],$ 
$$\mathbf{d}_s(i,j) = \sum_{i\in A, j\in B} w_s(A|B)$$ 
where the sum is over all splits of $s$ with $i$ in one part and $j$ in the other. The metric is often referred to as the distance vector $\mathbf{d}_s.$
 \end{definition}
 
It is well known that Kalmanson metrics are in one-to-one correspondence with weighted circular split networks. Specifically, from \cite{steelphyl}, and as repeated in \cite{scalzo}, we have the following:
\begin{lemma}\label{kalm}
A distance vector $\mathbf{d}$ is Kalmanson with respect to a circular order $c$ if and only if $\mathbf{d}$ = $\mathbf{d}_s$ for $s$ a unique weighted circular split system $s$, (not necessarily containing all trivial splits) with each split $A|B$ of $s$ having both parts contiguous in that circular order $c$.
\end{lemma}

\begin{definition}
We define the \textcolor{def}{\emph{minimum path distance}} vector $\mathbf{d}_N$ for a weighted 1-nested phylogenetic network $N,$ where 
$$\mathbf{d}_N(i,j) = \min_p\{\sum_{e \in p} w_N(e)~|~p  \text{ is a path connecting } i,j\}$$ where the minimum is over paths $p$ from leaf $i$ to leaf $j,$ and each sum is over edges in one of those paths. Examples are calculated in Figures~\ref{swlwuno} and~\ref{twoo}.
\end{definition}

\subsubsection{Resistance distance}\label{sec_weight} Now we define a new kind of pairwise distance functions on the leaves of a phylogenetic network.
Isolating sections of circuit-parallel paths between two leaves allows the Ohm relations, together with the $Y$-$\Delta$ transformation, to be used to find the effective resistance between those leaves. A simplifying fact is that the resistance between two leaves only depends on the resistances of edges that are in paths between those leaves. (We use the term \emph{pairwise circuit} $P_{ij}$ to refer to the edges that are in any path between leaves $i,j.$ For example see Figure~\ref{bigp}.)

There is a well-known alternate method for calculating effective resistances.  As defined in \cite{klein93} and \cite{bapat04}, the resistance distance matrix for a graph $G$ with $n$ total vertices (leaves and non-leaf nodes) is given by:

 $$\Omega_{ij}=\Gamma_{ii}^{-1}+\Gamma_{jj}^{-1}-2\Gamma_{ij}^{-1}$$ 
 
 where $\Gamma=L+1/n,$
the Laplacian matrix of $G$ plus the $n\times n$ matrix with $1/n$ for every entry.
Our resistance distance for phylogenetic networks uses entries of the matrix $\Omega$. 
\begin{definition}
We define the \textcolor{def}{\emph{resistance distance}} vector $\mathbf{d}_N^R$ for a positive weighted  phylogenetic network $N,$ where 
$\mathbf{d}_N^R(i,j)$ is the resistance distance on the graph between leaves $i$ and $j$. That is, $\mathbf{d}_N^R(i,j) = \Omega_{ij}$ for leaves $i$ and $j$. The distance can also be calculated using the basic relations of Ohm's law. Examples of the resistance distance vector are in Figures~\ref{s795},~\ref{ex23} and~\ref{level2resist}.\end{definition}

\subsubsection{Weighted Functions} We next define functions between the weighted split networks and the weighted phylogenetic networks. As previously explained in \cite{durell} and \cite{scalzo}, we begin by extending the function $L$ to a weighted version $L_w.$

\begin{definition}\cite{scalzo}
For a weighted circular split network $s$ we define $L_w(s)$ to be the 1-nested phylogenetic network $L(\overline{s})$ (the smoothed exterior subgraph of the unweighted version of $s$), with weighted edges. The weight of an edge in the image is found by summing the  weights of splits which contribute to that edge. Let $p_s(e)$ be the set of splits $A|B$ of $s$, such that $A|B$ is represented by edges in $s$ one of which is used to form the edge $e$ in $L(s)$. (Reccall that $e$ in $L(s)$ is formed by smoothing a path of edges from the exterior subgraph of $s$.) If $w_s$ is the weight function on $s$ then the weight function on $L_w(s)$ is:

$$w_{L_w(s)}(e) = \sum_{A|B \in p_s(e)} w_s(A|B).$$  
\end{definition}
By this definition we have the following (from \cite{scalzo}): \begin{lemma}\label{comm}
$\overline{L_w(s)} = L(\overline{s}).$
\end{lemma} 
For an example of $L_w$ see Figure~\ref{shortynum}.
\begin{figure}
    \centering
    \includegraphics[width=\textwidth]{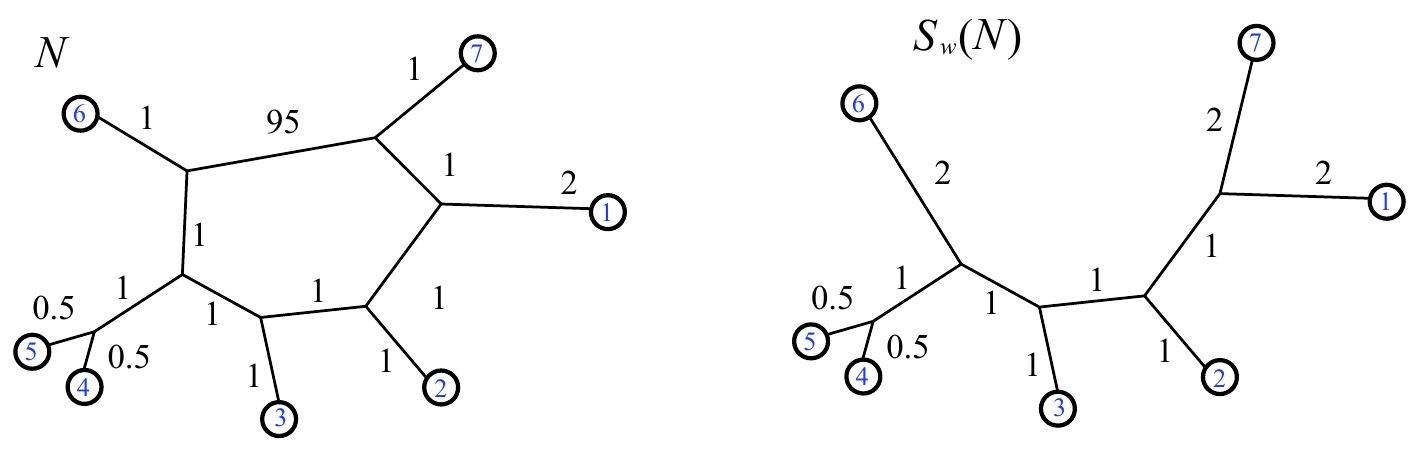}
    \caption{Example of the action of $S_w$. Here  $$\mathbf{d}_N = (4, 5, 6.5, 6.5, ~7, ~4,$$ $$\hspace{.652in}3, 4.5, 4.5, ~5, ~4,$$ $$\hspace{.5in}\hspace{.35in}3.5, 3.5, ~4, ~5,$$ $$\hspace{.5in}\hspace{.5in}\hspace{.2in}1, 3.5, 6.5,$$ $$\hspace{.35in}\hspace{.25in}\hspace{.25in}\hspace{.5in}3.5, 6.5,$$
    $$\hspace{.5in}\hspace{.5in}\hspace{.25in}\hspace{.5in}7).$$}
    \label{swlwuno}
\end{figure}
From \cite{scalzo}, we have the fact that the minimum path distance is Kalmanson for planar networks. Therefore, as in that source, we can make the following:

\begin{definition}
Given a weighted unrooted phylogenetic network $N$ that can be drawn on the plane with leaves on the exterior, we define $S_w(N)$ to be the unique weighted circular split network with the same minimum path distance vector as $N$. That is, $\mathbf{d}_N = \mathbf{d}_{S_w(N)}.$ This image is calculable, for instance, as the circular split network $S_w(N) = \mathcal{N}(\mathbf{d}_N),$ where $\mathcal{N}$ is the neighbor-net algorithm defined by \cite{Bryant2007} and implemented as in Splits-Tree\cite{Huson1998}.  Thus to find $S_w(N)$ we first calculate the minimum path distance vector, $\mathbf{d}_N$, and then use any algorithm (such as neighbor-net) to find the split network.
\end{definition}

For an example see Figure~\ref{swlwuno}. Another example of $S_w$, on a 2-nested network, is in Figure~\ref{shortynum}.
When we restrict to the domain of weighted circular split networks arising from weighted 1-nested networks, the codomain of $S_w$ is the outer-path circular split networks, and the distance vector is preserved by the map $L_w.$ Specifically from \cite{scalzo} we have: 
\begin{lemma}\label{lemw}
For any weighted 1-nested phylogenetic network $N$, if $s=S_w(N)$ then $s$ is outer-path and thus $\mathbf{d}_{L_w(s)} = \mathbf{d}_s.$
\end{lemma}

$S_w$ is defined using the minimum path distance metric. Similarly, since  we will see that the resistance distance is Kalmanson in Theorem~\ref{kal}, then by Lemma~\ref{kalm} we can make the following definition using resistance distance.
\begin{definition}
For a weighted 1-nested phylogenetic network $N$ we define $R_w(N)$ to be the unique weighted circular split network corresponding to the resistance distance $\mathbf{d}_N^R.$ This image is calculable, for instance, as the circular split network $R_w(N) = \mathcal{N}(\mathbf{d}^R_N),$ where $\mathcal{N}$ is the neighbor-net algorithm defined by \cite{Bryant2007} and implemented as in Splits-Tree\cite{Huson1998}. The algorithm neighbor-net is guaranteed to produce $R_w(N)$ using input  $\mathbf{d}_N^R.$
\end{definition}
There are several algorithms for finding the unique circular split system associated to a Kalmanson network; here we recommend neighbor-net and its implementation in  \cite{Huson1998}. That algorithm finds both the circular split network and its weighting. However, for a weighted 1-nested phylogenetic network $N$, due to our Theorems~\ref{kal} and~\ref{tsplits} we can calculate the weighted circular split network $R_w(N)$ directly, bypassing both the calculation of the metric and the use of neighbor-net. The function $R_w$ is shown by example in Figure~\ref{s795}. For another example, on a 2-nested network that happens to be Kalmanson, see Figure~\ref{level2resist}.

 \begin{figure}
    \centering
    \includegraphics[width=6.5in]{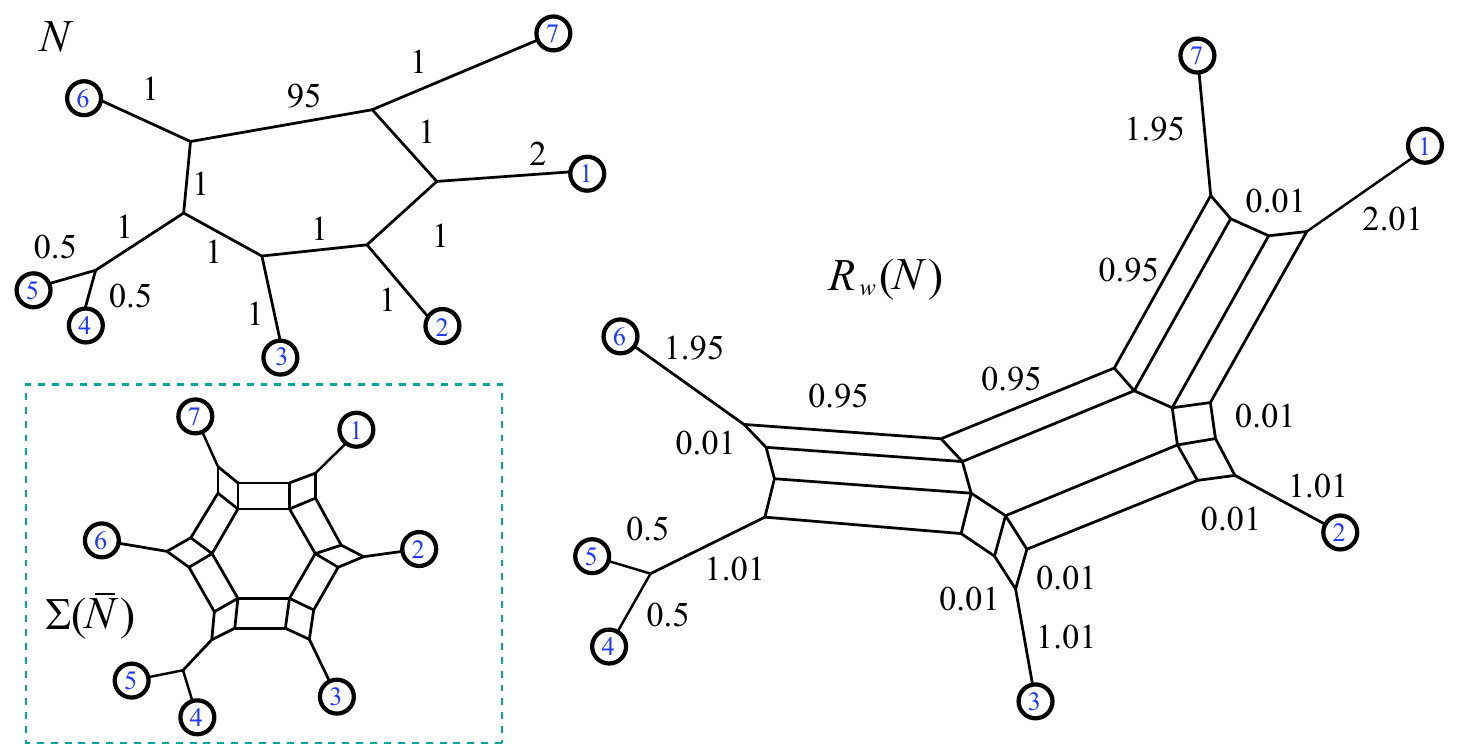}
    \caption{Example of the function $R_w$ which takes a weighted phylogenetic network and outputs the split network associated to its resistance distance. Neighbor-net can be run on input $\mathbf{d}_N^R$, or the results of Theorem~\ref{tsplits} can be used: for instance the value 0.95 for the split $\{1,2,3,7\}|\{4,5,6\}$ is found by 95(1)/(95+1+1+1+1+1). Here $$\mathbf{d}_N^R = (3.99, 4.96, 6.41, 6.41, 6.84, 3.99,$$ $$\hspace{.855in}2.99, 4.46, 4.46, 4.91, 3.96,$$ $$\hspace{.5in}\hspace{.75in}3.49, 3.49, 3.96, 4.91,$$ $$\hspace{.5in}\hspace{.5in}\hspace{.75in}1, ~3.49, 6.34,$$ $$\hspace{.5in}\hspace{.5in}\hspace{.5in}\hspace{.5in}3.49, 6.34,$$ 
    $$\hspace{1.75in}\hspace{.45in}\hspace{.25in}6.75).$$}
    \label{s795}
\end{figure}

\begin{remark}\label{coreflection}~
When restricted to phylogenetic trees, the functions $L_w$ and $S_w$ are both the identity, and $S_w=R_w$. In \cite{scalzo} several other properties are listed, in the process of showing that $L_w$ and $S_w$ form a Galois coreflection when restricted to weighted 1-nested  phylogenetic networks and outer-path circular split networks. These include the facts that $L_w$ is injective but not surjective,  $S_w$ is surjective but not injective, and  $S_w\circ L_w$ is the identity map.  
\end{remark}


\section{Kalmanson networks}\label{thms}

The main result in this section is that the resistance metric is Kalmanson for 1-nested phylogenetic networks, and that the unique associated split network has the same exterior form as the original 1-nested phylogenetic network.
 First we show that $\mathbf{d}_N^R$ obeys the Kalmanson condition: there exists a
circular ordering of $[n]$ such that for all $i < j < k < l $ in that ordering,

$$\max\{\mathbf{d}_N(i,j)+\mathbf{d}_N(k,l),\mathbf{d}_N(j,k)+\mathbf{d}_N(i,l)\} \le \mathbf{d}_N(i,k)+\mathbf{d}_N(j,l).$$

\begin{thm}\label{kal}
 Given a 1-nested phylogenetic network $N$ with positive weighted edges and $n$ leaves, the resistance metric on its leaves is Kalmanson.   
\end{thm}

\begin{proof}

  The cyclic order that we need to exist in order to demonstrate the Kalmanson property is found by choosing any cyclic order of $[n]$ consistent with $N$. That is, we choose an outer planar drawing of $N$ and use the induced cyclic order of the leaves arranged around the exterior of that drawing.
  
  Begin by noting that for each pair of the four leaves $i,j,k,l$ there is a sub-graph, called the \emph{pairwise circuit}, for instance $P_{ik}$, made of all the edges which are part of any path between those two leaves.  The pairwise circuit will contain perhaps some cycles---it will in fact be a series of cycles connected by paths. We are especially interested in the intersection $I$ of the two ``crossing'' pair circuits, $I$ = $P_{ik} \cap P_{jl}$.  There are three basic cases to consider.
  
  Case 1: The intersection $I$ is a single cycle.
  Here the four leaves $i,j,k,l$ have pairwise circuits that reach the cycle $I$ at four different nodes. Notice that any of the two pairwise circuits summed in the Kalmanson condition will include all four of the smaller pairwise circuits from each of the four leaves $i,j,k,l$ to the node  of $I$ closest to that respective leaf. We will call those closest nodes $v_i, v_j, v_k, v_l.$ The three sums in the Kalmanson condition all share some terms in common: those which come from the weighted edges in pairwise paths between the four leaves and the respective nodes $v_i, v_j, v_k, v_l.$ Discarding these common terms, we are left with terms that come from the weighted edges in $I$. Thus the only differences between the three sums in the Kalmanson condition arise from the different contributions of the cycle $I$.  We denote by $a, b, c, d$ the cumulative edge weights between the four nodes $v_i, v_j, v_k, v_l,$ following the cyclic order. For instance, in Figure~\ref{case1}, $a$ is the weight of the edge between $v_l$ and $v_i$ and $b$ is the sum of the weights on edges of $I$ between the nodes  $v_i$ and $v_j$.

  Thus we can write the sums explicitly: $$\mathbf{d}_N^R(i,j)+\mathbf{d}_N^R(k,l) = \mathbf{d}_N^R(i,v_i)+ \mathbf{d}_N^R(j,v_j)+\mathbf{d}_N^R(k,v_k)+\mathbf{d}_N^R(l,v_l)+ \frac{b(a+d+c)}{a+b+c+d}+\frac{d(a+b+c)}{a+b+c+d};$$

  $$\mathbf{d}_N^R(j,k)+\mathbf{d}_N^R(i,l)= \mathbf{d}_N^R(i,v_i)+ \mathbf{d}_N^R(j,v_j)+\mathbf{d}_N^R(k,v_k)+\mathbf{d}_N^R(l,v_l)+\frac{c(a+b+d)}{a+b+c+d}+\frac{a(b+c+d)}{a+b+c+d};$$
  
  $$\mathbf{d}_N^R(i,k)+\mathbf{d}_N^R(j,l)= \mathbf{d}_N^R(i,v_i)+ \mathbf{d}_N^R(j,v_j)+\mathbf{d}_N^R(k,v_k)+\mathbf{d}_N^R(l,v_l)+\frac{(a+d)(b+c)}{a+b+c+d} +\frac{(a+b)(c+d)}{a+b+c+d}.$$

   After discarding the common terms, we consider just the remaining sums of fractions. All the edge weights are positive, and the denominators of all three are the same. Clearly the third sum, when expanded, has a numerator larger than either of the first two.
   \begin{figure}
      \centering
       \includegraphics[width=5.25in]{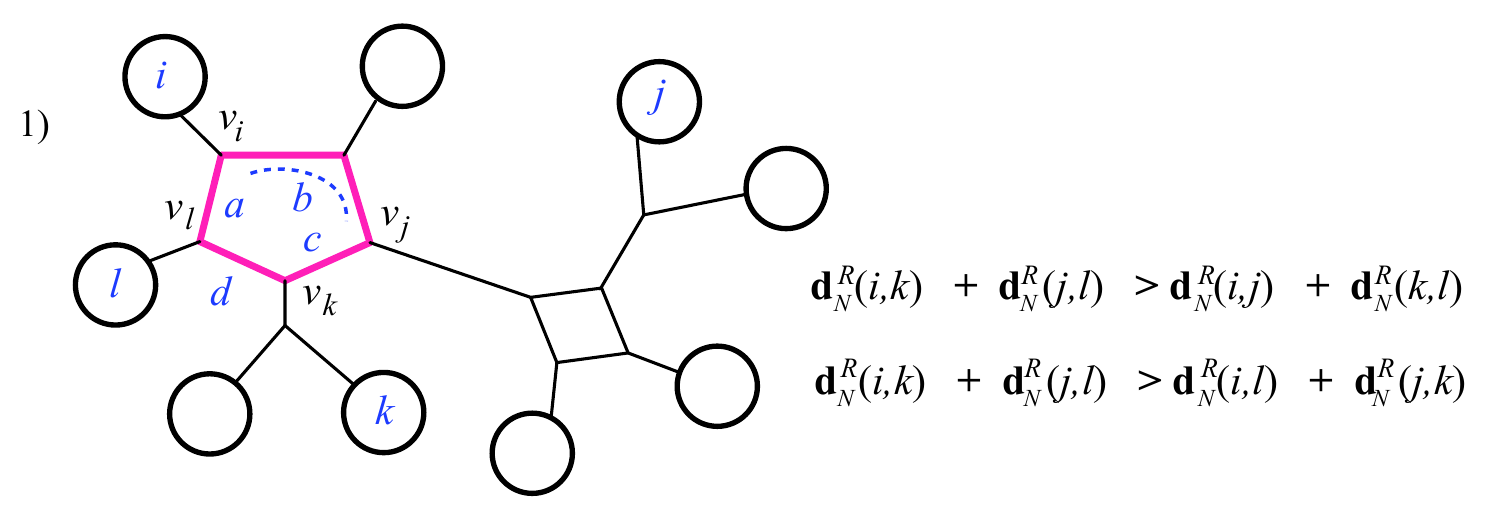}
      \caption{Case 1 of Theorem~\ref{kal}: the highlighted edges are the intersection $I$ of the pairwise circuits between leaves $i,k$ and $j,l.$}
      \label{case1}
  \end{figure}

  Case 2: The intersection $I$ is a series of cycles containing at least two cycles. In this case there are two possible ways that the inequalities are satisfied, depending on which pair of consecutive leaves ($i,j$ or $j,k$) reach the same end of $I$, that is, have their attaching nodes ($v_i,v_j$ or $v_k$) in $I$ at the same end of $I$. In Figure~\ref{case2} below we choose $i,j$ to do so, on the left-hand cycle, but the other option is similar. Checking this case can be done visually for the equality:  $\mathbf{d}_N^R(i,k)+\mathbf{d}_N^R(j,l) = \mathbf{d}_N^R(i,l)+\mathbf{d}_N^R(j,k)$ since the two sums end up using precisely the same effective resistances. That is,
  both $\mathbf{d}_N^R(i,k)+\mathbf{d}_N^R(j,l)$ and  $\mathbf{d}_N^R(i,l)+\mathbf{d}_N^R(j,k)$ have all terms in common: both the portions from the paths outside of $I$ as in case 1, and the summands contributed by $I$, which are the terms:
  $$\frac{c(a+b)}{a+b+c}+\frac{a(b+c)}{a+b+c}+\frac{x(w+y)}{w+x+y}+\frac{w(x+y)}{w+x+y}.$$
  
  The inequality $\mathbf{d}_N^R(i,k)+\mathbf{d}_N^R(j,l) > \mathbf{d}_N^R(i,j)+\mathbf{d}_N^R(k,l)$ (for the subcase where again $i,j$ reach the same end of $I$) is easily checked. Here, after discarding the terms in common, the larger sum contains more terms than the smaller (from the parts of $I$ not in the pairwise circuits for $i,j$ and $k,l$). As well, when the smaller sum has terms with denominator matching a term in the larger, the numerator is indeed larger in the latter. For instance, in Figure~\ref{case2}, after discarding the common terms contributed by the paths outside of $I$, the sum $\mathbf{d}_N^R(i,j)+\mathbf{d}_N^R(k,l)$ has the sum contributed by $I$: $$\frac{b(a+c)}{a+b+c}+\frac{y(w+x)}{w+x+y}.$$ The numerator here is exceeded by the sum contributed by $I$ in $\mathbf{d}_N^R(i,k)+\mathbf{d}_N^R(j,l)$ as just listed above. Finally, notice that there are sub-cases of Case 2 in which the smaller sum will have fewer or no terms at all contributed by $I$; these occur when $I$ includes a path at one end or at both ends. See Figure~\ref{case2b} for example. 
  \begin{figure}
      \centering
       \includegraphics[width=5.25in]{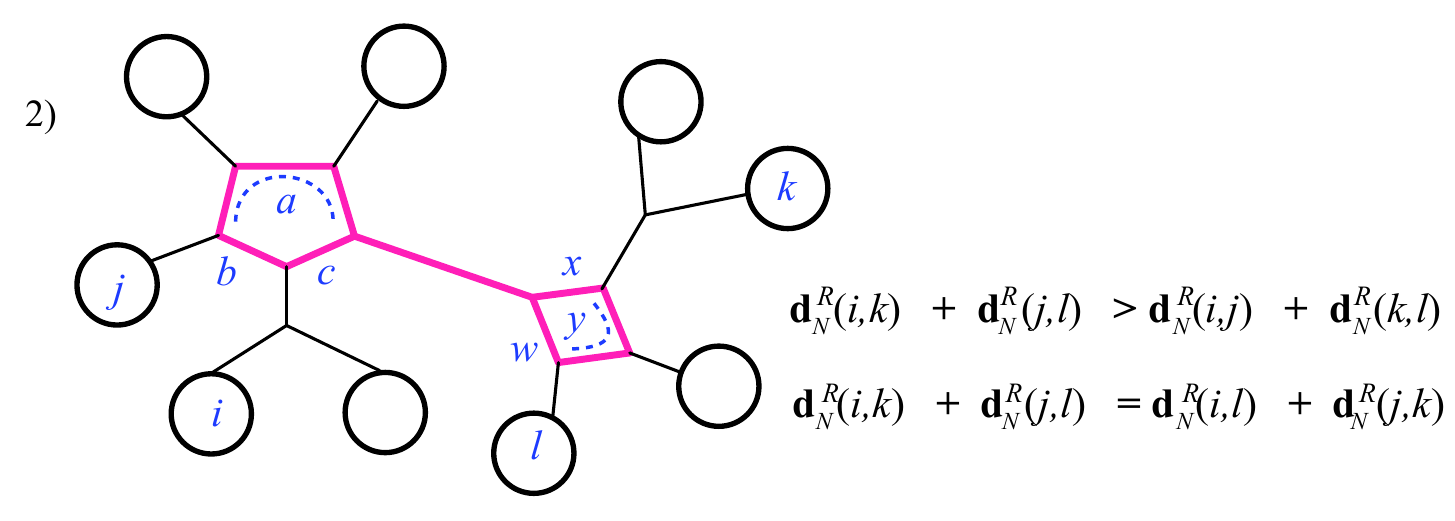}
      \caption{Case 2 of Theorem~\ref{kal}: the highlighted edges are the intersection $I$ of the pairwise circuits between leaves $i,k$ and $j,l.$}
      \label{case2}
  \end{figure}

  \begin{figure}
      \centering
       \includegraphics[width=5.25in]{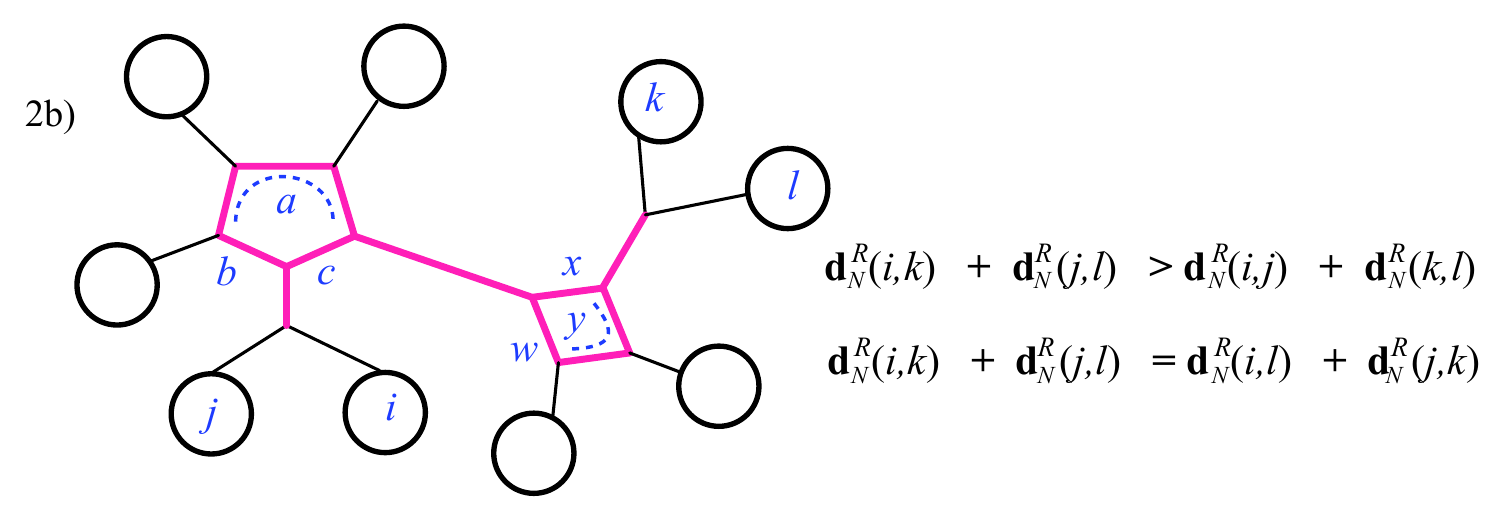}
      \caption{Case 2 of Theorem~\ref{kal} continued. Here $v_i=v_j$ and $v_k=v_l$.}
      \label{case2b}
  \end{figure}
  
  Case 3: The intersection $I$ is a path. In this case it is quickly verified that the Kalmanson inequality is satisfied as an equality. See Figure~\ref{case3} for example.
  
  \begin{figure}
      \centering
       \includegraphics[width=5.25in]{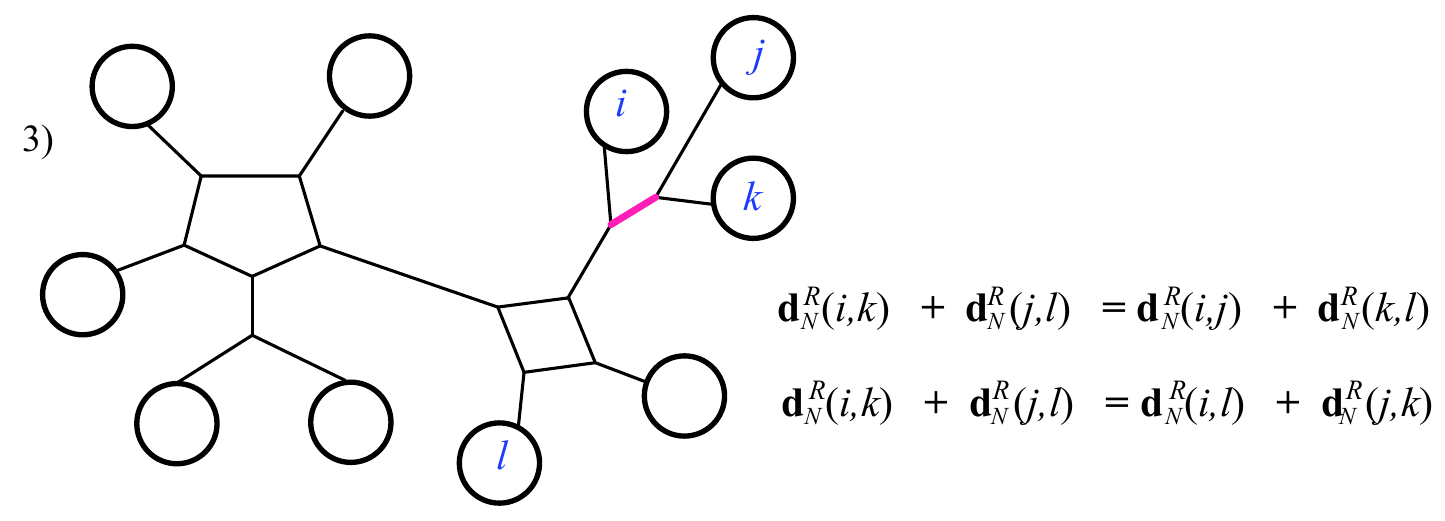}
      \caption{Case 3 of Theorem~\ref{kal}.}
      \label{case3}
  \end{figure}
\end{proof}

The fact that effective resistance distance is a Kalmanson metric immediately suggests that it would be a good candidate for modelling weighted phylogenetic networks. First there is the intuition from experience that if two pathways of heredity exist, the ancestor individual or species will have more in common with the extant individual or species.  Thus mutations in the genetic code play the role of resistors to the flow of information. 

Secondly, Kalmanson metrics are known to be the only example for which each metric is represented uniquely by a circular split system, as seen in Lemma~\ref{kalm}. In the case of the resistance distance, the associated unique split network has an additional advantage: it is guaranteed to represent faithfully every split displayed by the original 1-nested network. 
\begin{thm}\label{tsplits}
 Given a 1-nested phylogenetic network $N$ with positive weighted edges and $n$ leaves, and letting $\mathbf{d}_N^R$ be the resistance metric on the $n$ leaves, then the unique associated split network $R_w(N) = \mathcal{N}(\mathbf{d}_N^R)$ displays precisely the same splits as displayed by $N$.
\end{thm}
\begin{proof} 
 A split $A|B$ can be displayed by $N$ in three possible ways: either it is displayed by a single bridge $e$ with weight $w(e)$, by a pair of edges both in the same cycle $c$ with respective weights $a_c$ and $x_c$, or in more than one way.  Let the weight of a specific display of a split in $N$ be $w(e)$ in the first case and $(a_cx_c)/z_c$ in the second case, where $z_c$ is the sum of all the weights in the cycle. We claim: if the split  $A|B$ in  $\Sigma(\overline{N})$ is assigned the sum of the weights of all distinct displays of that split as displayed in $N$, then the resulting distance metric $\mathbf{d}$ from the weighted split network thus constructed is indeed $\mathbf{d}_N^R$. Therefore we will conclude, since Theorem~\ref{kal} shows that $\mathbf{d}_N^R$ is Kalmanson, that the weighted split network thus constructed is equal to the unique split network corresponding to $\mathbf{d}_N^R$, as found for instance by the algorithm neighbor-net.

 First we check that the claim holds. Consider the pairwise circuit $P_{ij}$ in $N$ for a given pair $i,j$ of leaves. It will be a series of paths and cycles, as seen for example in Figure~\ref{bigp}. 
 \begin{figure}
      \centering
       \includegraphics[width=5in]{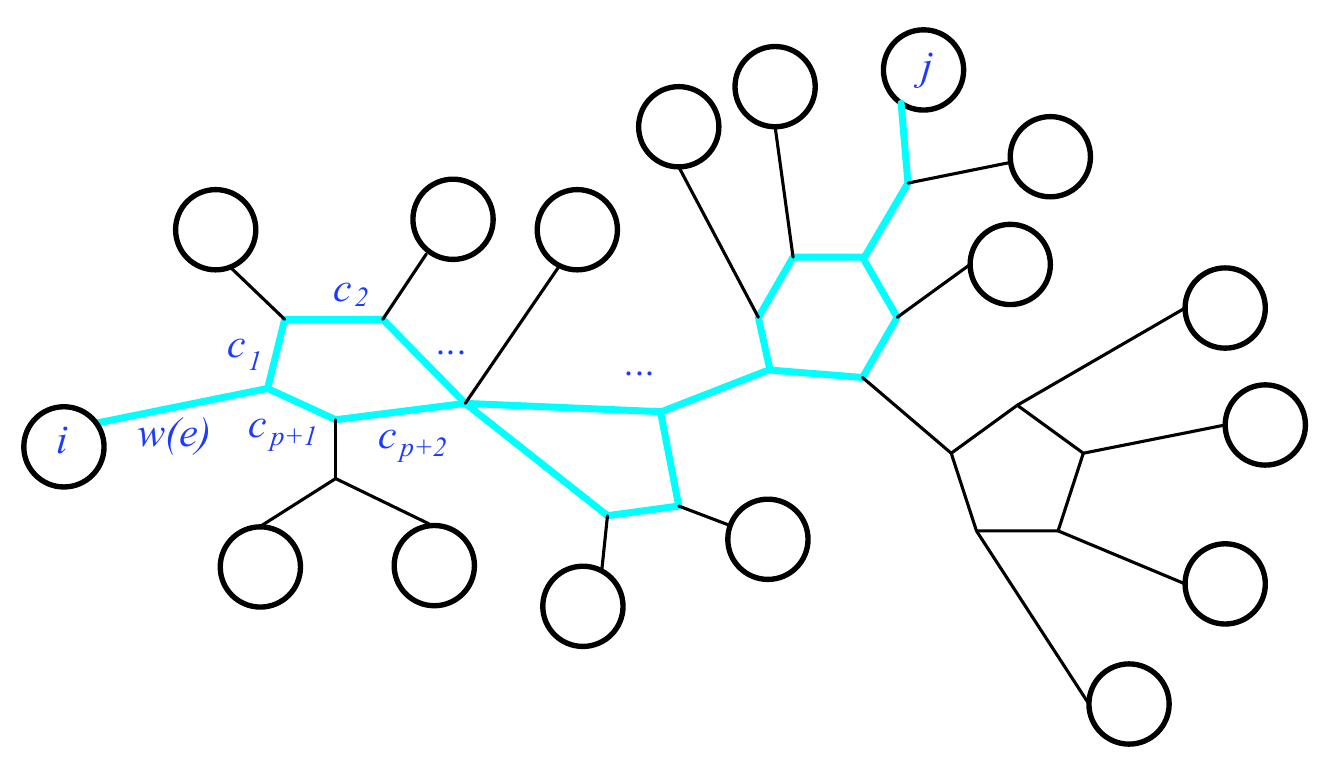}
      \caption{The highlighted subgraph is the pairwise circuit $P_{ij}.$}
      \label{bigp}
  \end{figure}
  Thus each cycle $c$ in $P_{ij}$ will be split into two circuit-parallel paths $p_c$ and $q_c$ of respective lengths $p,q$. Both paths begin and end at the two nodes where that cycle is attached to the rest of the series. Now the resistance distance $\mathbf{d}_N^R(i,j)$ will be the sum of the weights of the (non-circuit-parallel) paths, and of the effective resistances of the circuit-parallel paths. Specifically, every weighted edge of $P_{ij}$ not in a cycle will contribute its weight to the sum, and every weighted edge in a cycle of $P_{ij}$ will appear in one of two factors in the numerator of the term giving the effective resistance from those circuit-parallel paths. We see that
 
 $$\mathbf{d}_N^R(i,j) = \sum_{e\in P_{ij}}w(e) +\sum_{c\in P_{ij}}\frac{(c_1+\dots+c_p)(c_{p+1}+\dots+c_{p+q})}{c_1+\dots +c_{p+q}} = \sum_{e\in P_{ij}}w(e) +\sum_{c\in P_{ij}}\sum_{c_m \in p_c \atop c_r\in q_c}\frac{c_mc_r}{z_c}  $$ 
 
 where $c_1,\dots,c_p$ and $c_{p+1},\dots,c_{p+q}$ are the weights of the circuit-parallel paths of cycle $c \in P_{ij}$, with $z_c = c_1+\dots+c_{p+q}$ being the total weight of $c.$ That is, we expand the numerator of each term from a cycle. Now, the distance metric corresponding to the weighted split network we constructed using $\Sigma(\overline{N})$ has distance 
 
 $$\mathbf{d}(i,j)= \sum_{A|B \in N \atop i\in A, j\in B} w(A|B) $$ 
 Now splits in $N$, and thus in $\Sigma(\overline{N}),$ which separate leaves $i,j$ are precisely those displayed by a bridge in $P_{ij}$ or by a pair of circuit-parallel edges in  a cycle of $P_{ij}$. 
 Thus using the weights for splits (as stated above): 
 
 $$w(A|B)= \sum_{A|B \text{ disp. by } e} w(e) + \sum_{A|B \text{ disp. by } \atop c_m,c_r \in c} \frac{c_mc_r}{z_c}$$  in the split metric, gives us the desired claim: $\mathbf{d} = \mathbf{d}_N^R$.

 Then we conclude that since the  weighted circular split network associated to the original Kalmanson metric $\mathbf{d}_N^R$ is the unique such network where the split metric equals the original Kalmanson metric, then $\mathcal{N}(\mathbf{d}_N^R)$ will have precisely the splits of $N$ and thus of $\Sigma(\overline{N}).$ 
\end{proof}

Remark: The fact that we can take a weighted 1-nested phylogenetic network $N$ and build a weighted circular split network $s$ which has the same metric, $\mathbf{d}_s = \mathbf{d}_N^R$, implies another proof that the resistance distance is Kalmanson. Since the circular split network is planar, and the split metric on it is the same as the minimum path network on it, that metric is guaranteed to be Kalmanson. However, our original proof has the advantage that we see which of the inequalities are strict, and which are actually equalities.

The first important implication of these theorems is that the resistance distance on any 1-nested phylogenetic network $N$ is precisely represented by a unique circular split network $\mathcal{N}(\mathbf{d}_N^R)$. Exactly all the splits displayed by the original $N$ are present in $\mathcal{N}(\mathbf{d}_N^R)$. Thus the function $L$ applied to the unweighted version of $\mathcal{N}(\mathbf{d}_N^R)$ returns the unweighted version of $N$ itself.

\begin{thm}\label{net}
  Given weighted 1-nested $N$, we have that $\overline{\mathcal{N}(\mathbf{d}_N^R)} = \Sigma(\overline{N}).$ Thus $L(\overline{\mathcal{N}(\mathbf{d}_N^R))} = \overline{N}.$
\end{thm}
\begin{proof}
 The first equality follows directly from Theorem~\ref{tsplits}, since neighbor-net is guranteed to output the splits of the unique circular split network associated to the Kalmanson metric given by the resistance distance, which is indeed all the splits displayed by the network $N$. Then  from \cite{scalzo}, we have the second equality since $L\circ \Sigma$ is shown there to be the identity map. 
\end{proof}

The first application implied by this result is that when using neighbor-net on a measured distance matrix, if we assume that it reflects a resistance distance, we can always recover the form of the original network. The weights of splits in the result of neighbor net are interesting, they are in fact terms in the expansion of the calculated resistance distance. However, the first advantage we see is that the original unweighted phylogenetic network can be directly recovered by taking the exterior of the result of neighbor-net.

As an alternative to neighbor-net, there are polytopes which can serve as the domain for linear programming that finds  the best-fit 1-nested phylogenetic network for a measured distance matrix.

\section{Resistance distance and  polytopes}\label{sec_poly}
In \cite{durell} the authors describe a new family of polytopes. This family lies between the Symmetric Travelling Salesman Polytope (STSP($n$)) and the Balanced Minimum Evolution Polytope (BME($n$)). Our polytopes are called the level-1 network polytopes BME($n,k$) for $0\le k \le n-3$. All have  dimension ${n \choose 2} - n.$  In \cite{scalzo} we looked at implications of the Galois connections studied there for these polytopes, especially using $S_w$, the function based on minimum path distance.  It turns out that if we assume an input distance metric represents the resistance distance on a 1-nested phylogenetic network $N$, then the result of neighbor-net or of linear programming on a BME polytope is a network accurately showing all the splits of $N$. Also,  neighbor-net is statistically consistent, as shown in \cite{Bryant2007}. Therefore as a measured set of pairwise distances approach the resistance distance of $N$, the output of neighbor-net will approach the faithfully phylogenetic circular split network $N$.  This is in contrast to minimum path distance where some genetic connections are assumed to be negligible, and then are lost in the output of neighbor-net. However, the theorems about minimum path distance, specifically Theorems 8, 9 and 11 of \cite{durell}, play an important role in the proof of Theorem~\ref{face} here. Here we repeat some of the same introductory definitions and remarks and then extend the results to resistance distance.

\begin{definition}\label{e:bmenkvert} For a binary, 1-nested phylogenetic network $N$, (weighted or unweighted) the vector ${\mathbf x}(N)$ is defined to have lexicographically ordered components ${x}_{ij}(N)$ for each unordered pair of distinct leaves $i,j \in [n]$ as follows:

$$ {x}_{ij}(N) = \begin{cases} 2^{k-b_{ij}} & \text{if there exists cyclic order $c$ consistent with $N$; with $i,j$  adjacent in $c$,}\\ 0 & \text{otherwise.} \end{cases}
$$

where $k$ is the number of bridges in $N$ and $b_{ij}$ is the number of bridges traversed in a path from $i$ to $j$. For example, see Figure~\ref{srx}.
 \end{definition}
The  convex  hull  of  all  the   ${\mathbf x}(N)$ such that binary $N$ has $n$ leaves and $k$ nontrivial bridges is the level-1 network polytope BME($n,k$). As shown in \cite{durell}, the vertices of BME($n, k$) are precisely the vectors ${\mathbf x}(N)$ for $N$ binary with $n$ leaves and $k$ nontrivial bridges. In light of Theorems~\ref{kal} and~\ref{tsplits}, we can now characterize the vertices in terms of resistance distance:
\begin{thm}\label{vert}
Every 1-nested phylogenetic network found as an image $L(\overline{R_w(N)})$ gives rise to a face of BME($n,k$) for some $k.$ In particular, the vertices of the polytopes $BME(n,k)$ correspond to images  $L(\overline{R_w(N)})$ which exhibit $k$ non-trivial bridges, for weighted 1-nested networks $N$ with $n$ leaves and such that any node not in a cycle has degree three. 
\end{thm}
 \begin{proof} The image $R_w(N)$ will faithfully represent all splits, as seen in  Theorem~\ref{tsplits}. Thus $\overline{R_w(N)}$ will be faithfully phylogenetic, in the range of $\Sigma.$ Specifically,  the function $R_w$ will introduce bridges that separate all cycles, thus insuring that any node in a cycle will have degree three.  Therefore if the non-cycle nodes of $N$ are degree three, $L(\overline{R_w(N)})$ will be a binary unweighted 1-nested phylogenetic network.
 \end{proof}
Also as shown in \cite{durell} and repeated in \cite{scalzo}, an equivalent definition of the vector $\mathbf{x}(N)$ is the vector sum of the vertices of the STSP($n$) which correspond to cyclic orders consistent with $N$. Recall that the vertices of STSP($n$) are the incidence vectors $\mathbf{x}(c)$ for each cyclic order $c$ of $n$, where the $i,j$ component is 1 for $i$ and $j$ adjacent in the order $c$, 0 otherwise. This equivalent definition for binary 1-nested phylogenetic networks may also be applied to any 1-nested phylogenetic network:
\begin{lemma}\label{cyc}
For a 1-nested phylognetic network $N$, the vector $\mathbf{x}(N)$ is equal to $\sum_c \mathbf{x}(c)$ where the sum is over all cyclic orders $c$ of $[n]$ consistent with $N.$ 
\end{lemma}
We point out, for the sake of attribution, that for phylogenetic trees  $t$ (with nodes of any degree), Lemma~\ref{cyc} with $N=t$ gives a formula for $\mathbf{x}(t)$ that agrees with the definition of the coefficient $n_t$ in \cite{Steel}, in the proof of Theorem 4.2 of that paper.


\begin{figure}
    \centering
    \includegraphics[width=5in]{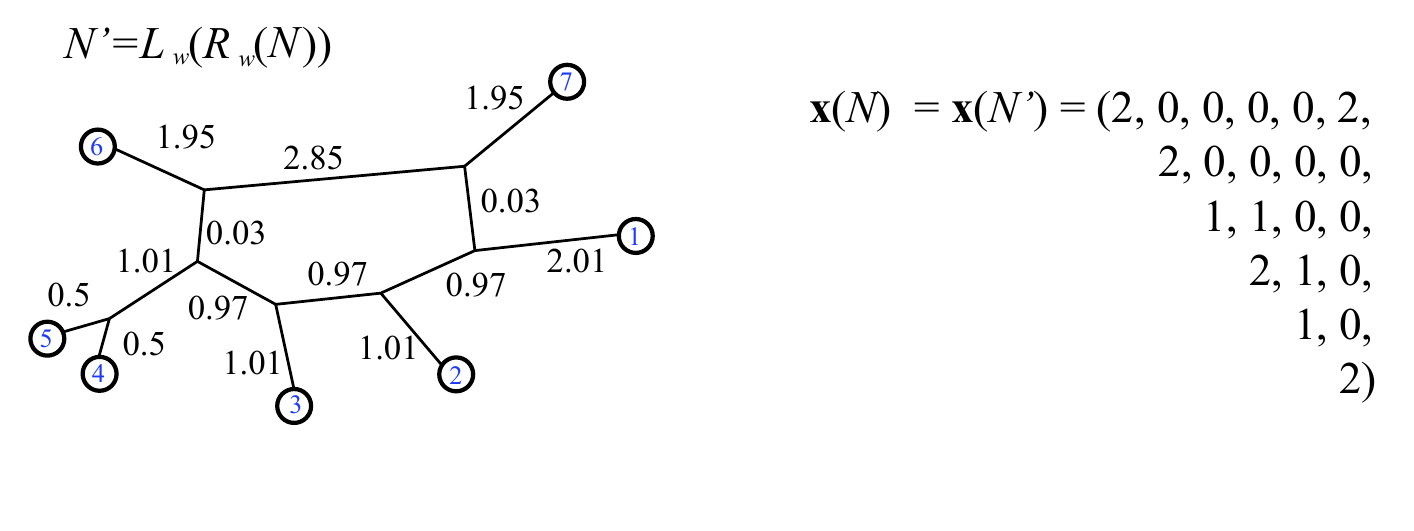}
    \caption{Using $N$ from Figure~\ref{s795}, we find $N'$ as in the proof of Theorem~\ref{face}. In the vector $\mathbf{x}(N)$ the first component is $x_{1,2} = 2^{1-0},$ since there are no non-trivial bridges traversed between leaves 1 and 2. As well, there are two consistent circular orders, with leaves 1 and 2 adjacent, found by twisting around the single non-trivial bridge.  The 19th entry is $x_{5,6} = 2^{1-1},$ since the path between leaves 5 and 6 traverse the non-trivial bridge. Here the minimum path distance vector is: $$\mathbf{d}_N = (3.99, 4.96, 6.43, 6.43, 6.84, 3.99,$$ 
    $$\hspace{.855in}2.99, 4.46, 4.46, 4.93, 3.96,$$
    $$\hspace{.5in}\hspace{.75in}3.49, 3.49, 3.96, 4.93,$$ $$\hspace{.5in}\hspace{.5in}\hspace{.75in}1, ~3.49, 6.34,$$ $$\hspace{.5in}\hspace{.5in}\hspace{.5in}\hspace{.5in}3.49, 6.34,$$ 
    $$\hspace{1.75in}\hspace{.45in}\hspace{.25in}6.75 ).$$}
    \label{srx}
\end{figure}

In \cite{scalzo} it is shown that the minimum path distance vector for a 1-nested phylogenetic network may be seen as a linear functional, and that it is minimized over the BME($n,k$) polytope. Specifically,
\begin{thm}\label{newth}
 Given any weighted 1-nested phylogenetic network ${N}$ with $n$ leaves, the product $\mathbf{x}(\hat{N})\cdot \mathbf{d}_{{N}}$ is minimized over BME($n,k$)  precisely for the unweighted binary 1-nested networks $\hat{N}$ with $k$ bridges such that $\overline{S_w({N})} \le \Sigma(\hat{N})$.\\
\end{thm}
Here $\hat{N}$ is used to denote a variable binary 1-nested phylogenetic network, taking values from the set of networks which refine $\overline{S_w({N})}$.  By this refinement we mean taking values from the set of networks with a superset of the set of splits displayed by $\overline{S_w({N})}.$
Now we can extend that result to resistance distances. In fact it becomes stronger: binary networks can be directly recovered even when they have long edges, since the action of $R_w$ preserves all splits. Precisely, we have:
\begin{thm}\label{face}
The minimum of $\mathbf{x}(N)\cdot\mathbf{d}_N^R$ is achieved at the face of BME($n,k$) with vertices $\mathbf{x}(\hat{N}),$ for unweighted binary networks  $\hat{N}$  with $k$ bridges such that $\hat{N}$ refines $\overline{N}.$
\end{thm} 
\begin{proof}
We claim that $\mathbf{x}(N)\cdot\mathbf{d}_N^R$ is the same as $\mathbf{x}(N)\cdot\mathbf{d}_N'$ for $N'=L_w(R_w(N))$. That is because the leaves which are adjacent in some circular order consistent with $N$ and thus in $R_w(N)$ have distance between them which is the sum of the splits that separate them. Since those leaves are adjacent, the shortest path of splits between them will lie on the exterior of $R_w(N)$. In fact, for adjacent $i,j$ an edge of a cycle on the path between them with weight $a$, contributes $ \frac{a(b+c+d+...)}{a+b+c+d+...}$ to $\mathbf{d}_N^R(i,j)$, where the other edges of that cycle have weights $b, c, d, ....$ Bridges $e$ between them contribute their weights $w(e)$. These values are the same as those for the splits displayed between $i,j$, seen in the proof of Theorem~\ref{tsplits}. Therefore:
\begin{align*}\mathbf{x}(N)\cdot\mathbf{d}_N^R &=\sum_c\mathbf{x}(c)\cdot \mathbf{d}_n^R\\
&=\sum_c\mathbf{x}(c)\cdot \mathbf{d}_s, \text{ for } s=R_w(N)\\
&=\sum_c\mathbf{x}(c)\cdot \mathbf{d}_{N'} \text{ for } N'=L_w(R_w(N))\\
&=\mathbf{x}(N)\cdot\mathbf{d}_{N'}
\end{align*}

 We know from Theorems 8, 9, and 11 of \cite{durell} that for any weighted 1-nested phylogenetic network ${M}$ with $n$ leaves, the product $\mathbf{x}(\hat{M})\cdot \mathbf{d}_{{M}}$ is minimized over BME($n,k$) precisely for binary networks $\hat{M}$ with $k$ bridges such that $\overline{M} \le \hat{M}.$ 
 
Thus in our case we have $\mathbf{x}(\hat{N})\cdot \mathbf{d}_{{N'}}$ is minimized over BME($n,k$)  precisely for the unweighted binary networks $\hat{N}$ with $k$ bridges such that $\overline{{N'}} \le \hat{N}$. Here, $\overline{(N')} = \overline{N},$ since $\overline{L_w(R_w(N))} = \overline{N}.$ 
The inequality here is refinement.
\end{proof}


For example compare Figures~\ref{s795} and~\ref{srx}. It is easily checked that although $\mathbf{d}_{{N'}} \ne \mathbf{d}_N^R$, we have  $\mathbf{x}({N})\cdot \mathbf{d}_{{N'}} = \mathbf{x}(N)\cdot \mathbf{d}_N^R = 51.4$

The implication then is that using either linear programming on BME($n,0$) or neighbor-net, assuming that the resistance metric is valid, the resulting split network gives the true exterior form of the original 1-nested phyologenetic network.

\section{2-nested networks, Counterexamples and Conjectures}\label{counter}

 In this section, we examine functions between 1-nested and 2-nested networks, and circular split networks.  We point out how well the various distance measurement distinguish or do not distinguish between network types, via examples. Then we make some conjectures based on observations. 
 
 \subsection{2-nested networks}\label{too}
Towards the end of \cite{gambette-huber}, the authors ask: is it possible to characterize split systems induced by more complex uprooted networks such as 2-nested networks (i.e., networks obtained from 1-nested networks by adding a chord to a cycle)?  At first we interpret this question to be about the result of applying $S_w.$ That is, we specialize the question to asking more specifically which kinds of split systems correspond to  2-nested networks, via assigning them a weighting, finding the minimum path distance, and then finding the unique corresponding circular split network?  The question is still open, but we begin by carefully defining 2-nested networks and making some initial observations.

\begin{definition} 
For $N$ an unrooted phylogenetic network, if every edge of $N$ is part of at most two cycles, we call it a 2-nested network. By this definition, 2-nested networks contain 1-nested networks as a subset, which in turn contain 0-nested networks, which are phylogenetic trees.  By strict $k$-nested networks we mean $k$-nested but not $(k-1)$-nested. We will add the extra descriptor of triangle-free-ness explicitly when desired.
\end{definition}

A weighted 2-nested network is shown in Figure~\ref{twoo}, with its minimum path distance vector. 
\begin{figure}[h]
\centering
\includegraphics{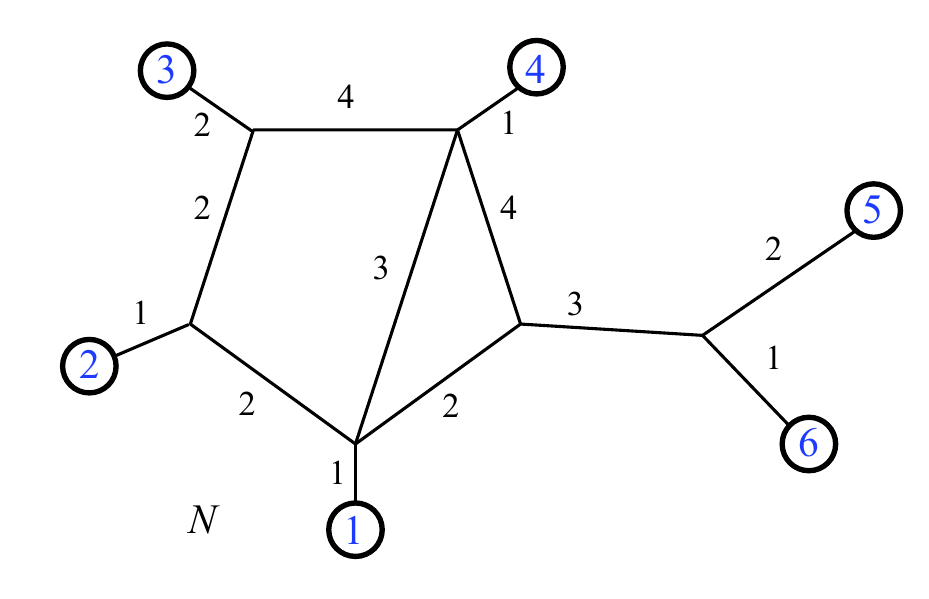} 
\caption{The minimum path distance vector for the weighted 2-nested network $N$  is 
$\mathbf{d}_N = ( 4, 7 , 5 , 8 , 7 , 5 , 7 , 10 , 9 , 7 , 13 , 12 , 10 , 9 , 3 ) $ .
Note that $\mathbf{d}_{14} = 5$, for example, referring to the shortest distance between leaves 1 and 4.}\label{twoo}
\end{figure}

\noindent
The first case we note is that weighted 2-nested networks often have images under $S_w$ that are not outer-path circular split networks. For instance see Figure~\ref{shortynum}. Therefore, by Lemma~\ref{outer},  2-nested networks can lead to split networks distinct from those induced via $S_w$ from 1-nested networks.   Also, applying $S_w$ and then $L_w$ in sequence will produce a weighted 1-nested network that has a different distance vector than the original. 

\begin{figure}
    \centering
    \includegraphics{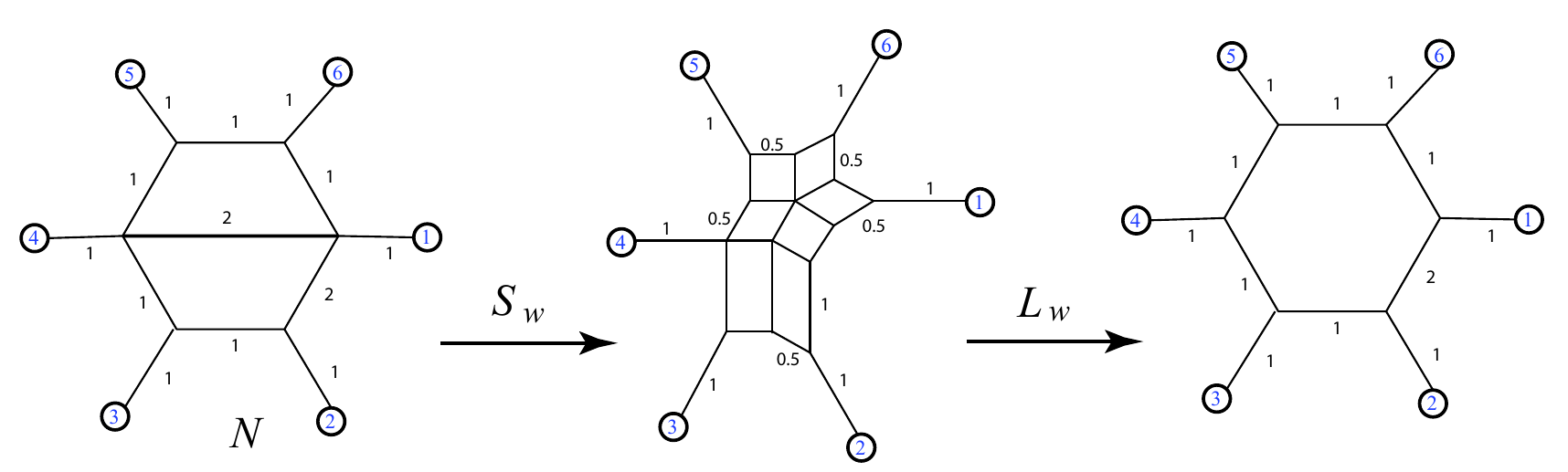}
    \caption{Here the output of $S_w(N)$ is a non-outer-path circular split network, and its image under $L_w$ has a distance vector that does  not match the original: for instance $\mathbf{d}_N(1,4) = 4$ but the distance from 1 to 4 in $L_w(S_w(N))$ is 5.}
    \label{shortynum}
\end{figure}

However, not all weighted 2-nested networks lead to distinct images from the 1-nested networks, under $S_w.$ In fact we have the following:
\begin{thm}For every weighted 1-nested network M, there exists some (not unique) weighted 2-nested network N such that the minimum path distance vectors coincide: $\mathbf{d}_M = \mathbf{d}_N$.\end{thm}

\begin{proof}
Consider a 1-nested network $M$ with positive values for its edges and a 2-nested network $N$ that has the same exterior subgraph. Let $N$ also have the same positive values for its exterior edges, but a positive value for its internal chord large enough such that on paths of least distance the internal chord of the 2-nested network is never used. Therefore both networks will have the same distance vector $\mathbf{d}_M = \mathbf{d}_N$.
\end{proof}


\subsubsection{Counting 2-nested networks}
We begin counting the total number of unweighted binary, triangle free, 2-nested networks. The numbers of unweighted binary, triangle free, 2-nested networks exist with $n$ leaves are: 6, 120, 2790 for $n=4,5,6.$  
    
First, consider structures with 4 leaves ($n=4$). We start by considering the unlabeled pictures, and then count the ways to assign the values $1,\dots,4$ to the leaves. In fact, we can simplify further by finding the unlabelled 1-nested networks and showing the potential locations of chords simultaneously in each picture. There is one such  unlabeled picture for $n=4$ as shown in Figure~\ref{five}, with two possible internal chords. There are $\frac{3!}{2}$ ways to arrange the leaves before choosing a chord. Therefore, the total number of unweighted binary triangle-free 2-nested networks with $n=4$ leaves is $(2)\frac{3!}{2} = 6$.


\begin{figure}[h]
\centering
\includegraphics[scale=1]{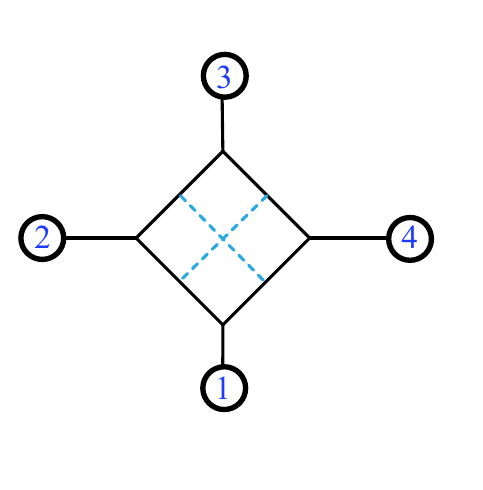} 
\includegraphics[scale=1]{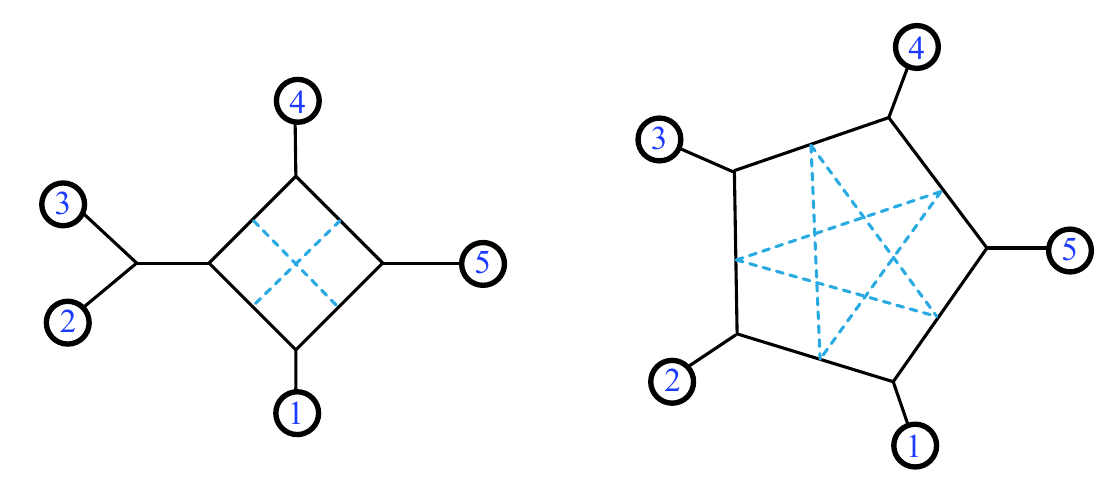} 
\includegraphics[scale=.75]{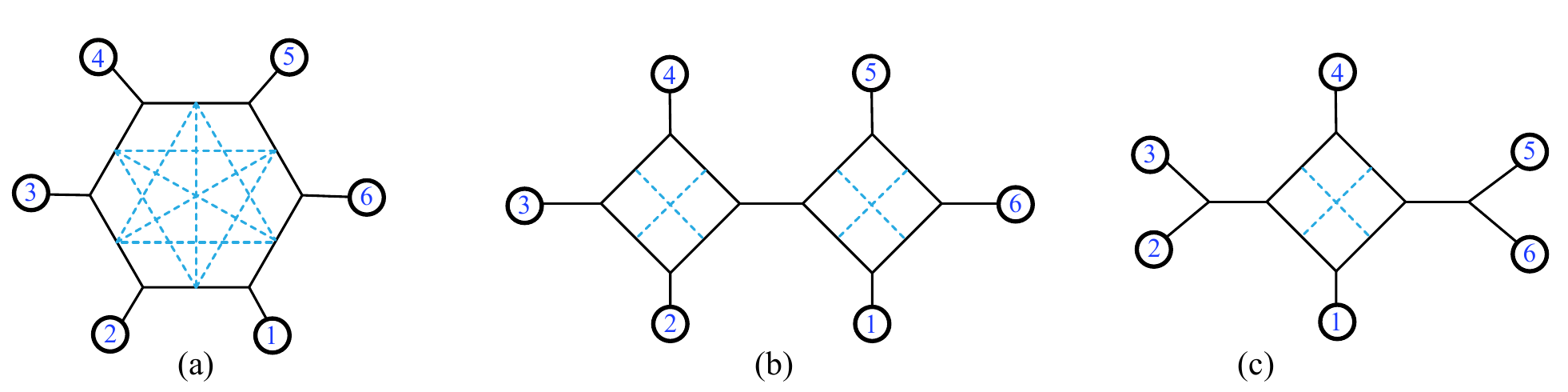} 
\includegraphics[scale=.75]{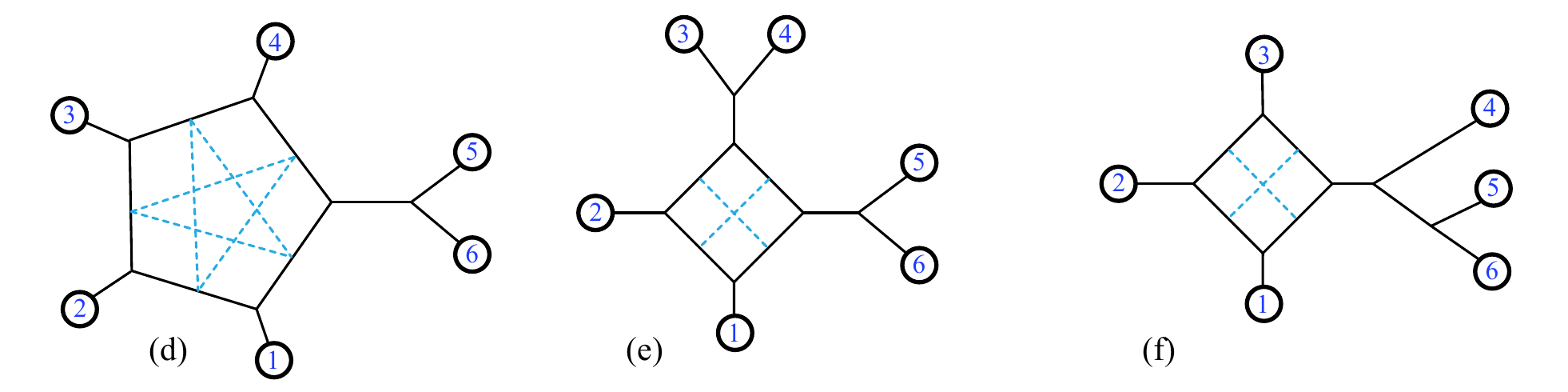}
\caption{For $n=4$, there is only one exterior structure  with two internal  chords possible (as seen by the dotted lines above). For $n=5$, there exist two exterior structures. For $n=6$ there are 6 such structures, labelled a-f.}\label{five}
\end{figure}

For $n=5$ the possible internal structures are shown in Figure~\ref{five}. There are 5 possible internal chords for one structure, and 2 possible internal chords for the other. The number of ways to arrange the leaves of the first structure is $n!$, and the second structure is $(n-1)!$ (since the first is not rotationally symmetric.) However, rearranging the leaves clockwise and counterclockwise yield the same rearrangement, so we must then divide by 2 to eliminate half of the arrangements garnered from the counting of those leaves. Finally, if there were a bridge connecting any components of the structure, simply divide by 2 for the twisting around that bridge. 
\noindent The counting for each $n=5$ structure in Figure~\ref{five} is as follows:
$$\frac{5(2)}{2}\frac{4!}{2} = 60,$$
$$\frac{4(1)}{2}\frac{5!}{2}\frac{1}{2} = 60.$$
The total number of networks for $n=5$ is = 60 + 60 =120.\newline

\noindent For $n=6$ the counting for each structure is as follows (from a to f as pictured in Figure~\ref{five}):
\begin{center}
\begin{tabular}{ccc}
(a)\vspace{.1in}&&$\frac{6(3)}{2}\frac{5!}{2} = 540,$\\

(b)\vspace{.1in}&&$(2)(2)\frac{4(1)}{2}\frac{6!}{2}\frac{1}{2}\frac{1}{2} = 720,$\\

(c)\vspace{.1in}&&$\frac{4(1)}{2}\frac{6!}{2}\frac{1}{4}\frac{1}{2} = 90,$\\

(d\vspace{.1in})&&$\frac{5(2)}{2}\frac{6!}{2}\frac{1}{2} = 900,$\\

(e)\vspace{.1in}&&$\frac{4(1)}{2}\frac{6!}{2}\frac{1}{4}\frac{1}{2} = 180,$\\

(f)&&$\frac{4(1)}{2}(6!)\frac{1}{4}\ = 360.$\\
\end{tabular}
\end{center}
The total number of networks for $n=6$ is = 540 + 720 + 90 + 900 + 180 + 360 = 2790.
Notice for (f), reading the labels clockwise is not equivalent to reading them counterclockwise due to the tree structures. This means we just consider $6!$ and not $\frac{6!}{2}$. We ask whether there is a general formula for the number of binary triangle-free 2-nested networks with $n$ leaves. Aternatively, we might look for a 2-variable formula. In \cite{durell} there is a 2-variable formula for  binary triangle-free 1-nested networks with $n$ leaves and $k$ non-trivial bridges, which may serve as a model: $$ 
{n-3 \choose k}\frac{(n+k-1)!}{(2k+2)!!}.
$$ 

 \subsection{Indistinguishable weightings} Resistance distance metrics on a 1-nested phylogenetic network are not in bijection with edge weightings, but the  split-equivalence class is an invariant of those edge weights. That is, if two networks $N$ and $N'$ have the same resistance distance metric $\mathbf{d}^R_N = \mathbf{d}^R_{N'}$, this does not imply that $N=N'$, but it does imply that $\overline{N}=\overline{N'}.$ The latter fact is implied by Lemma~\ref{kalm} and the theorems of Section~\ref{thms}, and we can see the former fact via counterexample. In Figure~\ref{ex23} we show two weighted phylogenetic networks with 4 leaves, called $N$ and $N'.$ Their resistance distances between leaves are identical: $$\mathbf{d}_N^R = \mathbf{d}_{N'}^R = \left(\frac{122}{23},\frac{178}{23},\frac{108}{23},\frac{198}{23},\frac{168}{23},\frac{176}{23}\right).$$
Note that we do see that $\overline{N}=\overline{N'}.$ There are 7 split-classes of 1-nested phylogenetic networks on 4 leaves, and our theorems  show that none of the other 6 classes can be given edge weights that yield this same resistance distance metric on four leaves.

\begin{figure}
    \centering
    \includegraphics[width = 4.5in]{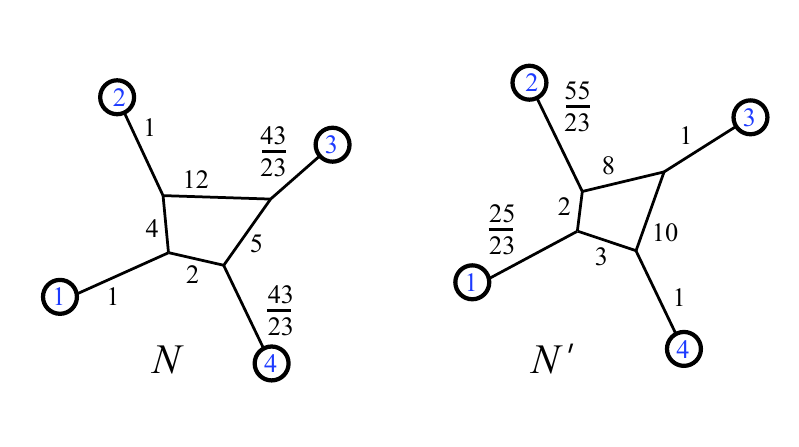}
    \caption{Two weighted phylogenetic networks with identical resistance distances for their leaves.}
    \label{ex23}
\end{figure}

\subsection{Non-Kalmanson networks} Not all resistance distances are Kalmanson, even when restricted to phylogenetic networks. For a counterexample, consider the network $N$ formed by having 6 leaves attached to the 6 vertices of the complete bipartite graph $K_{3,3},$ pictured in Figure~\ref{bip}.
\begin{figure}
    \centering
    \includegraphics[width = 2.25in]{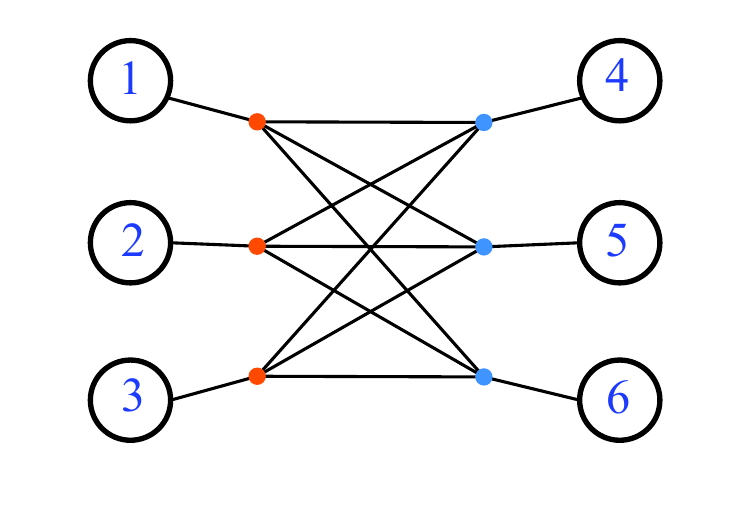}
    \caption{A phylogenetic network with non-Kalmanson resistance distance. All the edge lengths are 1.}
    \label{bip}
\end{figure}
The resistance distance metric for complete bipartite graphs is found in \cite{klein93}. Consider that $K_{m,n}$ is the graph join if two edgeless graphs: $K_{m,n} = \overline{K}_m + \overline{K}_n$ with unit weight for each edge. Then the resistance distance on $K_{m,n}$ is $2/n$ for vertices that have no edge between them (they are both the same color), and $(m+n-1)/mn$ for vertices with an edge between them \cite{klein93}. For our example $N$, let the two (same-colored) parts of the graph (3 nodes each, say red and blue) be attached to the leaves $\{1,2,3\}$ and $\{4,5,6\}$ respectively. Letting each edge have weight 1, we find the resistance distance between any two leaves attached to the same colored part is $2+2/3 = 8/3$, while the distance between any two leaves, with one attached to each part, is $2+5/9 = 23/9.$ In any circular order of the leaves, there will be a sub-sequence $i,j,k,l$ where the first two leaves $i,j$ are attached to the same color, and the second two $k,l$ are both attached to the other color. Thus $\mathbf{d}_N^R(i,j)+\mathbf{d}_N^R(k,l) = 16/3 = 48/9$ which is larger than $\mathbf{d}_N^R(i,k)+\mathbf{d}_N^R(j,l) = 46/9$.  
This counterexample raises the question of necessary conditions for a network with resistance distance to be Kalmanson.  

\subsection{Outer Planarity}\label{outer} We conjecture that outer planarity is a sufficient condition for Kalmanson: that if a weighted phylogenetic network can be drawn in the plane with its leaves on the exterior that the resistance distance is Kalmanson. We note that it this condition is not necessary: it can be checked that the complete graph $K_5$ with unit edges has the Kalmanson property.

\subsection{Faithfully phylogenetic Kalmanson distance vectors}\label{faith} Following the terminology in Definition~\ref{sig}, we call a Kalmanson distance vector $\mathbf{d}$ \emph{faithfully phylogenetic} if the unique  circular split network associated to $\mathbf{d}$  is in the range of  $\Sigma$  (after forgetting weights). We conjecture that faithfully phylogenetic Kalmanson distance vectors always arise from resistance distances. Specifically we conjecture that if $\mathbf{d}$ is faithfully phylogenetic, then $\mathbf{d} = \mathbf{d}_N^R$ for some weighted phylogenetic network $N.$
Note that not all Kalmanson distance vectors arise from resistance distances, simply due to the fact that not all circular split networks are in the range of $\Sigma$.

\subsection{2-nested Kalmanson networks} A special case of~\ref{faith} is the  conjecture that 2-nested phylogenetic networks have Kalmanson resistance distance.  For instance in Figure~\ref{level2resist} we show a simple 2-nested network $N$ whose resistance distance is clearly Kalmanson: in fact it is the same resistance distance as possessed by the shown 1-nested network. 

\subsection{Indistinguishable weightings and invariants} We conjecture that for every weighted 2-nested network there is a weighted 1-nested network with matching resistance distance. Again see Figure~\ref{level2resist}. However, in light of the above conjecture~\ref{outer}, we conjecture that the exterior shape of networks is an invariant of resistance distance: specifically that if any two outer planar networks $N,N'$ have $\mathbf{d}_N^R = \mathbf{d}_{N'}^R$  then $L(\overline{R_w(N)})=L(\overline{R_w(N')}).$

\begin{figure}
    \centering
    \includegraphics[width = \textwidth]{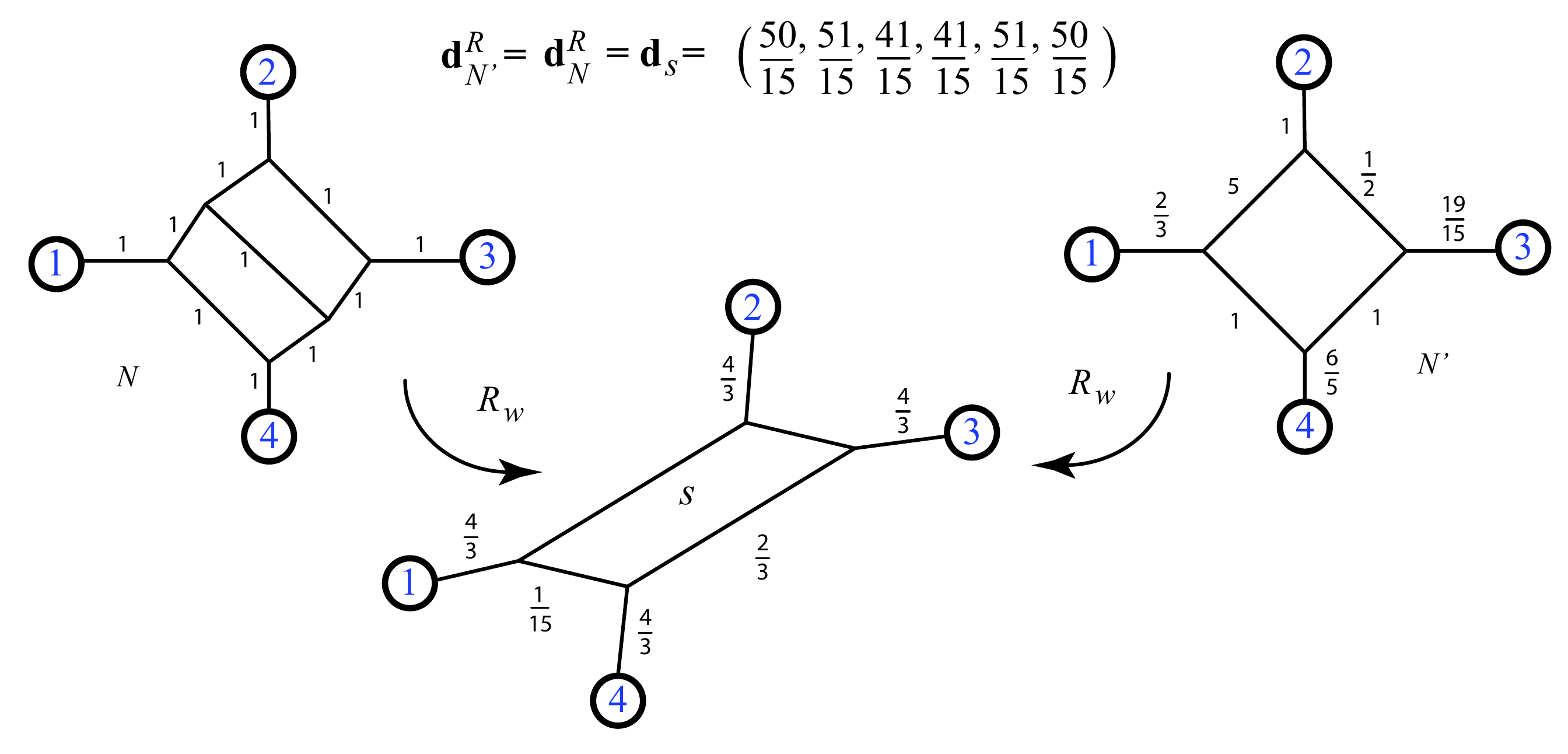}
    \caption{Two weighted phylogenetic networks with identical resistance distances for their leaves, and their common split network.}
    \label{level2resist}
\end{figure}

\subsection{Limiting case} Consider when an edge in a cycle of $N$ has a very large weight, or high resistance. As this weight grows, the limit of $L_w(R_w(N))$ approaches a network with that edge being deleted entirely. We see this by considering any two circuit-parallel paths with resistance $R_1$ and $R_2$ the first of which uses an edge with variable weight $w$ (all other weights constant). Then letting $w\to\infty$ implies $R_1\to\infty$ and thus $R_1R_2/(R_1 +R_2)$ approaches $R_2$ by L'Hospital's rule. Thus as $w$ goes to $\infty$ we see that the resistance distances using those circuit-parallel paths reduce to the path distances,  and  so the distance metric from that network approaches one without that edge. This is similar to the way in which $S_w$, which uses the minimal path distance on $N$, serves to delete some edges as seen in Figure~\ref{s795}.

\section{Distance measures}\label{jukes}
A question is raised about the mathematics which precedes the work described in this paper: what sort of measurement should actually yield the experimental resistance distances in a real example? What should play the role of attaching the ohmmeter to pairs of wires? Usually, DNA sequences of length $m$ are aligned (a multi-step problem of its own) and then the number of disagreeing sites is counted. Let $p$ be the proportion of disagreements to the length $m$ of the sequence: $p=(m-c)/m$ where $c$ is the number of correct, matching sites. Then there is a selection of mutation models, such as the simplest Jukes-Cantor model, which predict a distance $D$ which is the expected total number of mutations.  Experimentally we find that distance $D$ as a function of the observed disagreements.  Alternately we could choose $D$ from the list of evolutionary models: for instance $${\displaystyle D=K=-{1 \over 2}\ln((1-2p-q){\sqrt {1-2q}})}$$ for Kimura's two parameter model. Or, alignment-free models such as the $k$-mer  distance measures as described in \cite{seth}.

Here, we would want a distance $D=R$ which is summed when in sequence but obeys the Ohm equations. The answer will depend both on the model of mutation we choose and the model of recombination we choose. For instance, 
$\displaystyle{D=-{3 \over 4}\ln({1-{4 \over 3}p})}$  for the Jukes-Cantor model, as described in \cite{jukes}.
Rewriting using  $p=(m-c)/c$ we have:
$$
D(c)=\frac{3}{4}\ln\left(\frac{3m}{4c-m}\right).
$$
 
\begin{figure}
    \centering
    \includegraphics{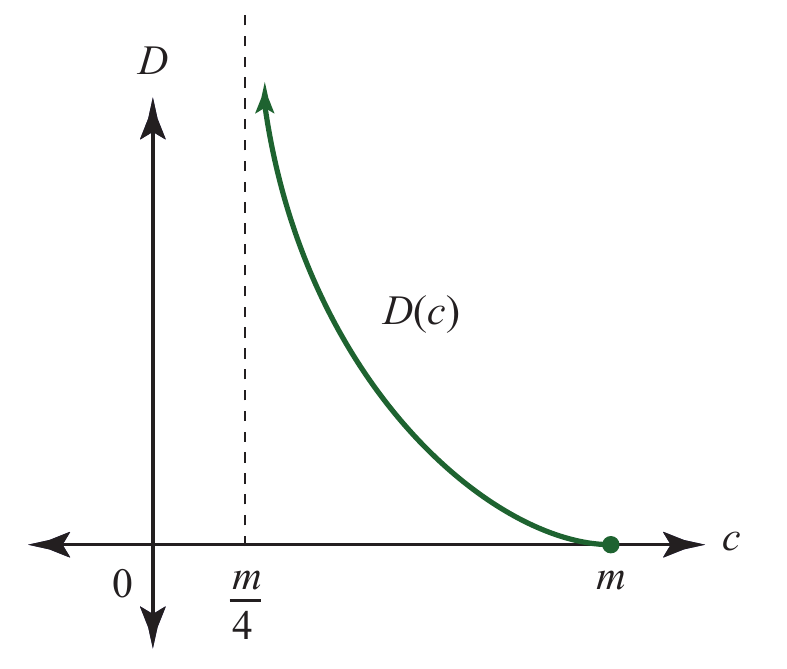}
    \caption{Calculated Jukes-Cantor distance $D$ as a function of the number of matching sites $c$ in aligned sequences of length $m$.}
    \label{Dgraph}
\end{figure}

$D$ has the graph in Figure~\ref{Dgraph}. The $c$-axis is explained by the fact that in the Jukes-Cantor model, mutations of the 4 nucleotides $A,G,T,C$ can replace any letter with another---including a self replacement. This implies that the smallest number of matching sites is $\frac{m}{4}$, while the largest is $m.$ We can use $D$ for the resistance distance only if there is experimental evidence that for circuit-parallel paths we have $D={D_1D_2}/({D_1+D_2}),$  where $D_1(c_1)$ and $D_2(c_2)$ are the distances for each path, in expected numbers of mutations as a function of correct matching sites. There are certainly some features of $D$ that look promising, including the shape of its graph: resistance typically ranges from 0 to infinity. Assuming that the formula for $D$ over the circuit-parallel paths does hold, when one  of the circuit-parallel resistances is infinite: say $D_1 \to \infty$; then we see that $D \to D_2$. Similarly, as  $c_1\to m/4$, we have that $c,$ the number of correct sites after recombination, approaches $c_2.$ 
 
 When both branches have the same distance $D_1=D_2$, and it obeys Ohm's law, we see the total resistance $D=D_1/2.$ Using the formula for $D(c)$ and $D_1(c_1)$  and solving for $c$ we get the following function, graphed in Figure~\ref{cgraph}: $$c = \frac{m}{4}+\sqrt{3\left(\frac{m}{4}c_1-\left(\frac{m}{4}\right)^2\right)}.$$ Thus as a first check the geneticist could compare two genomes and their hybrid genome with a common ancestor. When the two are close to the same distance from the common ancestor (both have $c_1$ matching sites), then the pair $(c_1,c)$  for $c$ the number of matches between the hybrid and the common ancestor might fit the parabola as seen in Figure~\ref{cgraph}. If that fit is achieved, then it would be reasonable to apply the theorems of this paper.

\begin{figure}
    \centering
    \includegraphics{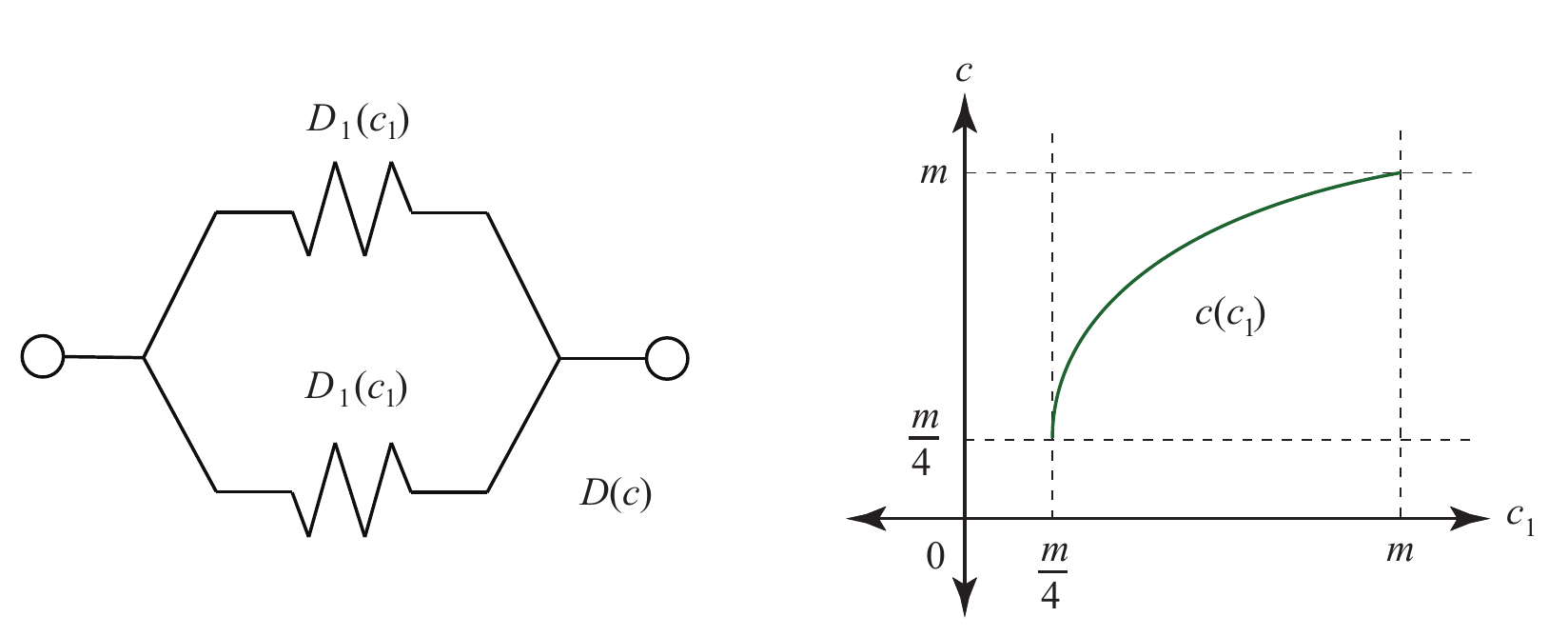}
    \caption{On the left is a simple parallel circuit with identical resistance on each branch. If the resistance is the Jukes-Cantor distance and obeys the Ohm laws, then the number $c$ of matching sites at the end of the circuit will depend on the number $c_1$ of correct matching sites at the end of each branch before recombination.}
    \label{cgraph}
\end{figure}

\section{Acknowledgements}

This manuscript has been released as a pre-print at arxiv.org/abs/2007.13574, \cite{pre}. We are thankful for proofreading by our referees, and for conversations with Jim Stasheff and Robert Kotiuga.

\bibliographystyle{amsplain}
\bibliography{phyloresist}{}

\end{document}